\documentclass[letter,11pt]{amsart}
\usepackage[latin1]{inputenc}   
\usepackage{amssymb,amsmath,amsthm,mathrsfs,mathdots}
\usepackage{graphicx,mathpazo}
\usepackage[all]{xy}
\usepackage[usenames,dvipsnames]{xcolor}
\usepackage[shortlabels]{enumitem}
\usepackage{enumitem}
\usepackage[normalem]{ulem}
\usepackage{youngtab,ytableau}
\ytableausetup{centertableaux,boxsize=0.25em}

\usepackage[colorlinks=true,linkcolor=Cerulean,pagebackref,citecolor=Maroon]{hyperref}

 \usepackage{tikz}
\usetikzlibrary{arrows,decorations.pathmorphing,decorations.pathreplacing,positioning,shapes.geometric,shapes.misc,decorations.markings,decorations.fractals,calc,patterns}

\tikzset{>=stealth',
     cvertex/.style={circle,draw=black,inner sep=1pt,outer sep=3pt},
     vertex/.style={circle,fill=black,inner sep=1pt,outer sep=3pt},
     star/.style={circle,fill=yellow,inner sep=0.75pt,outer sep=0.75pt},
     tvertex/.style={inner sep=1pt,font=\criptsize},
     gap/.style={inner sep=0.5pt,fill=white}}

\newcommand{\ZZ}{\ensuremath{\mathbb{Z}}}
\newcommand{\CC}{\ensuremath{\mathbb{C}}}

\newcommand{\NN}{\ensuremath{\mathbb{N}}}

\newcommand{\tr}{\operatorname{\mathfrak{T}}\nolimits}

\DeclareMathOperator{\Aut}{Aut}

\DeclareMathOperator{\End}{End}

\DeclareMathOperator{\GL}{GL}

\DeclareMathOperator{\Hom}{Hom}
\DeclareMathOperator{\id}{id}
\DeclareMathOperator{\Ind}{Ind}
\DeclareMathOperator{\Irrep}{Irr}

\DeclareMathOperator{\Mod}{\ensuremath{ \mathbf{Mod}}}

\DeclareMathOperator{\Res}{Res}

\DeclareMathOperator{\sgn}{sgn}

\DeclareMathOperator{\SL}{SL}
\DeclareMathOperator{\stn}{stn}

\DeclareMathOperator{\Sym}{Sym}
\DeclareMathOperator{\trv}{trv}
\DeclareMathOperator{\triv}{trv}

\newcommand{\sfG}{\ensuremath{\mathsf G}}
\newcommand{\sfH}{\ensuremath{\mathsf H}}
\newcommand{\shift}[1]{\ensuremath{\circlearrowleft#1}}   %mathfrak

\newcommand{\CM}{\operatorname{{MCM}}}

\newcommand{\Mat}{\mathrm{Mat}}

\def\xto{\xrightarrow}

\theoremstyle{theorem}
\newtheorem{Thm}{Theorem}[section]         %f"ur Satz1, 2 ,3, etc.
\newtheorem{lemma}[Thm]{Lemma}
\newtheorem{cor}[Thm]{Corollary}
\newtheorem{corollary}[Thm]{Corollary}
   
\newtheorem{prop}[Thm]{Proposition}
\newtheorem{proposition}[Thm]{Proposition}

\newtheorem{Qu}[Thm]{Question}

\newtheorem*{ThmIntro}{Theorem}

\theoremstyle{remark}
\newtheorem{Bem}[Thm]{Remark}
\newtheorem{ex}[Thm]{Example}%[chapter] 
\newtheorem{example}[Thm]{Example}
\theoremstyle{definition}
\newtheorem{defi}[Thm]{Definition} %neu
\newtheorem{notation}[Thm]{Notation}
\newtheorem{sit}[Thm]{}
\def\lto{{\longrightarrow}}
\newcommand{\vp}{\varphi}
\title{McKay quivers and Lusztig algebras  of some finite groups}

\author[Ragnar-Olaf Buchweitz]{Ragnar-Olaf Buchweitz$^\dagger$}
\address{Dept.\ of Computer and Math\-ematical Sciences,
University  of Tor\-onto at Scarborough, 
1265 Military Trail, 
Toronto, ON M1C 1A4,
Canada}

\author{Eleonore Faber}
\address{
School of Mathematics, University of Leeds, LS2 9JT Leeds, UK
}
\email{e.m.faber@leeds.ac.uk}

\author{Colin Ingalls}
\address{
School of Mathematics and Statistics,
Carleton University, 
Ottawa, ON K1S 5B6,
Canada}
\email{cingalls@math.carleton.ca}

\author{Matthew Lewis}
\address{
Department of Mathematics and Statistics,
University of New Brunswick, 
Fredericton, NB. E3B 5A3, 
Canada}
\email{Matthew.Lewis@unb.ca}

\thanks{The first and third authors were partially supported by an NSERC Discovery grant.  The second author gratefully acknowledges support by the European Union's Horizon 2020 research and innovation programme under the Marie Sk{\l}odowska-Curie grant agreement No 789580.}
\thanks{${}^\dagger$The first author passed away on November 11, 2017.}
\thanks{$^*$Corresponding author: e.m.faber@leeds.ac.uk} 
\thanks{Availability of data and material: computations with Macaulay2 available on request.}

\date{\today}

\subjclass[2010]{05E10 %Combinatorial aspects of representation theory
16G20 %representations of quivers
16S35 %Skew group rings
16S37 %Quadratic and Koszul algebras
20F55 %Reflection and Coxeter groups
20C30 %representations of finite symmetric groups
} 
\keywords{complex reflection groups, Young diagrams, Clifford theory, McKay quiver, skew group rings, Koszul algebras}

\begin{document}

\maketitle

\begin{abstract}
We are interested in the McKay quiver $\Gamma(\sfG)$ and skew group rings $A*\sfG$, where $\sfG$ is a finite subgroup of  $\mathrm{GL}(V)$, where $V$ is a finite dimensional vector space over a field $K$, and $A$ is a $K-\sfG$-algebra. These skew group rings appear in Auslander's version of the McKay correspondence.  

 In the first part of this paper we consider complex reflection groups $\sfG \subseteq \mathrm{GL}(V)$ and find a combinatorial method, making use of Young diagrams, to construct the McKay quivers for the groups $G(r,p,n)$. We first look at the case $G(1,1,n)$, which is isomorphic to the symmetric group $S_n$, followed by $G(r,1,n)$ for $r >1$. Then, using Clifford theory, we can determine the McKay quiver for any $G(r,p,n)$ and thus for all  finite irreducible complex reflection groups up to finitely many exceptions. 

In the second part of the paper we consider a more conceptual approach to McKay quivers of arbitrary finite groups: we define the Lusztig algebra $\widetilde A(\sfG)$ of a finite group $\sfG \subseteq \mathrm{GL}(V)$, which is Morita equivalent to the skew group ring $A*\sfG$.  This description gives us an embedding of the basic algebra Morita equivalent to $A*G$ into a matrix algebra over $A$.  
\end{abstract}

%\tableofcontents

\section{Introduction}

In 1979 John McKay noted an astonishing relationship between finite subgroups $\sfG \subseteq \SL(2,\CC)$ and Kleinian surface singularities $\CC^2/\sfG$ \cite{mckay37graphs}: via a directed graph (aka McKay quiver) constructed entirely from the irreducible representations of $\sfG$ one can recover the geometry of the minimal resolution of the quotient singularity $\CC^2/\sfG$, i.e., the dual resolution graph of $\CC^2/\sfG$. For these singularities, the dual resolution graphs are of Dynkin type $ADE$, and the corresponding McKay quivers have as their underlying graphs extended Dynkin diagrams of corresponding types. 
Since then, much work has been put into understanding and explaining this correspondence from various points of view, and in generalizing it to higher dimensions, see \cite{BuchweitzMFO} for an introduction and extensive references.  \\
Although the McKay quivers are known for small finite subgroups $\sfG \subseteq \GL(2, \CC)$ \cite[Prop.~7]{AuslanderReiten} and $\sfG \subseteq \GL(3,\CC)$  \cite{HuJingCai}, in general not much is known about the McKay quiver of any finite subgroup $\sfG \subseteq \GL(n,\CC)$ and the relation to the geometry of $\CC^n/\sfG$. 

Let $K$ be an algebraically closed field with $|G| \in K^*$.
In this paper, we consider the problem of describing the McKay quiver for a finite  reflection group $\sfG \subseteq \GL(n,K)$. Finite reflection groups have been completely classified in characteristic zero, by \cite{ShT}, also see \cite{arjeh1976finite}: a finite irreducible complex reflection group $\sfG$ is either one of $34$ exceptional groups (in which case $2 \leq n \leq 8$), or is a member of the infinite family $G(r,p,n)$. If $n=1$, then $G(m,1,1) \cong \ZZ/m\ZZ$. For $r=1$, $G(1,1,n)$ corresponds to the symmetric group $S_n$ (note that these are not irreducible, but can be seen so as operating on the invariant hyperplane ${x_1+\cdots+x_n=0}$ by permuting variables). Otherwise, to be irreducible, the parameters $r,p,n \in \NN_{>0}$ have to  satisfy $r>1$, $p|r$ and $(r,p,n) \neq (2,2,2)$, see \cite[Theorem 2.4]{arjeh1976finite}. The group $G(r,1,n)$ is isomorphic to $\mu_r \wr S_n$ and for $p >1$ the group $G(r,p,n)$ is a normal subgroup of index $p$ of $G(r,1,n)$. %(for a more detailed description see Section XXX). REFS? KANE? 

There is also a complete classification of irreducible finite reflection groups for arbitrary fields, see, for example, 
\cite{Mit1,Mit2,Mit3, Wag1, Wag2, ZSe}. These lists contain as well modular reflection groups. These examples 
and their invariant rings are further discussed in \cite{Bro} and \cite{KMa}.
Note that there are many infinite families that are not necessarily of the form $G(r,p,n)$.

Our main result regarding McKay quivers is a complete description of the McKay quiver of $\sfG=G(r,p,n)$: the vertices correspond to certain ($r$-tuples of) Young diagrams and the arrows can be interpreted as moving boxes in (the components of the $r$-tuples of) the Young diagrams: 

\begin{ThmIntro} (see Theorems \ref{McKayquiverSn}, \ref{thm:ML}, and \ref{Thm:McKayGrpn} for the detailed versions) Let $K$ be an algebraically closed field with $|G| \in K^*.$  The McKay quiver $\Gamma(\sfG)$ for $\sfG=G(r,p,n) \subseteq \GL(n,K)$ is given as follows:  
\begin{enumerate}[leftmargin=*]
\item $\sfG=G(1,1,n)$: The irreducible representations of $G(1,1,n)$ corresponding to the vertices of $\Gamma(\sfG)$ are indexed by the partitions of $n$. Let $\tau \neq \lambda$ be partitions of $n$. Then there is an arrow from $\lambda$ to $\tau$ if and only if $\lambda$ can be formed from $\tau$ by moving a single block. The number of loops on $\lambda$ is $p(\lambda)-1,$ where $p(\lambda)$ is the number of distinct parts of $\lambda$. 
\item $\sfG=G(r,1,n)$ with $r >1$: The vertices of $\Gamma(\sfG)$ correspond to $r$-tuples $\lambda=(\lambda^{(1)}, \ldots, \lambda^{(r)})$ of Young diagrams of size $0 \leq n_i \leq n$, such that $\sum_{i=1}^rn_i=n$. Then there is an arrow from a vertex $\alpha$ to  vertex $\beta$ in $\Gamma(\sfG)$ if and only if the $r-$tuple of Young diagrams $\beta$ can be obtained from $\alpha$ by deleting a cell from position $i$ in $\alpha$ and then adding a cell to position $i+1 \mod r$.
\item Let $\sfG=G(r,p,n)$ with $p >1$ and $p |r$. Then the irreducible representations of $\sfG$ correspond to tuples $([\lambda], \delta)$, where $\lambda$ is an irreducible representation of $G(r,1,n)$ and $\delta \in \mu_{u(\lambda)} \subseteq \mu_p$ (explained in Section \ref{Sub:IrrepGrpn}). Let $([\alpha],\delta)$ and $([\beta],\delta')$ be two vertices of $\Gamma(\sfG)$. Then there is an arrow with source $([\alpha],\delta)$ and target $([\beta],\delta')$ whenever $[\beta]$ can be obtained from $[\alpha]$ by moving a single cell contained in the \emph{fundamental domain} of $[\alpha]$ (see Def.~\ref{Def:FundDomain}) cyclically to the right. 
\end{enumerate}
\end{ThmIntro}
To show this, we use  Ariki--Koike's \cite{ariki1994hecke} description of the irreducible representations of $G(r,1,n)$, and we use Clifford Theory  (see \cite{stembridge1989eigenvalues}) to obtain the  irreducible representations of $G(r,p,n)$.

As a notable consequence, one sees that these McKay quivers may contain loops and multiple arrows (cf.~Thm.~\ref{McKayquiverSn}, Cor.~\ref{Cor:loopsmultiple}).
and are less symmetric, see Prop.~~\ref{symGroup}. %This behaviour did not occur for the case $\sfG \subseteq \SL(n,\CC), n\leq 3$. 

We want to mention here, that similar problems on combinatorics of Young diagrams and decomposition of powers of representations have been considered in \cite{GoupilChauve}, \cite{BritnellWildon}.

We are interested in complex reflection groups, since in \cite{BFI} the first three authors established a McKay correspondence for true reflection groups, i.e., groups generated by complex reflections of order $2$: the nontrivial irreducible representations of such a group $\sfG$ were shown to correspond to maximal Cohen-Macaulay modules over the discriminant of the reflection group, in particular, they correspond to the isotypical components of the hyperplane arrangement, viewed as a module over the coordinate ring of the discriminant, see \cite[Thm.~4.17]{BFI}.  Denoting $S=\Sym_{\CC}(K^n)$, this correspondence is realized by the isomorphism of a quotient of the skew group ring $S*\sfG$ with an endomorphism ring over the coordinate ring of the discriminant $S^{\sfG}/(\Delta)$. See also \cite{BFI-ICRA} for a more leisurely introduction. \\
This result is of the same flavor as  Auslander's  algebraic McKay correspondence \cite{Auslander86}: he showed that for a small subgroup $\sfG \subseteq \GL(n, K)$ the skew group ring $S*\sfG$ is isomorphic to $\End_R(S)$, where $R=S^\sfG$ is the invariant ring. For $n=2$, Auslander's result yields a correspondence between the isomorphism classes of indecomposable maximal Cohen--Macaulay $R$-modules and irreducible representations of $\sfG$. In particular, the so-called Gabriel quiver (the quiver of the algebra $S*\sfG$, cf.~\cite[Ch.~5, \S 3]{LeuschkeWiegand}) corresponds to the McKay quiver of $\sfG$ as well as to the Auslander--Reiten quiver of the category of maximal Cohen--Macaulay modules of $S^{\sfG}$ for these groups.

This leads to the second part of the paper: the motivation was to better understand the skew group rings $S*\sfG$. In the classical case, when $\sfG \subseteq \SL(2,K)$, then $S*\sfG$ is Morita equivalent to the preprojective algebra of an extended Dynkin diagram, see \cite{ReitenVdB}. In the literature, it has been studied when $S*\sfG$ is Morita equivalent to a preprojective algebra, see e.g.~\cite{Ringel, CB1999} and more recently, when it is Morita equivalent to a higher preprojective algebra (introduced in \cite{IyamaHigher}): see \cite{AmiotIyamaReiten, Thibault}. Moreover, for $S * \sfG$  one can construct a Morita equivalent path algebra with relations via super potentials, see \cite[Thm.~3.2]{BocklandtSchedlerWemyss} and also \cite{Karmazyn}, and compute the relations in this basic algebra. 

From a geometric point of view it is interesting to study $S*\sfG$ and the McKay quiver of $\sfG$ as well.
Since the category of $S*G$-modules is equivalent to the category of $G$-equivariant $S$-modules, smooth Deligne--Mumford stacky surfaces with smooth complete local coarse moduli spaces and trivial generic stabilizers are locally described by faithful linear actions by pseudo-reflection groups.  The vertices of the McKay quiver describe the local isomorphism classes of irreducible equivariant vector bundles, i.e., vector bundles on the stack, and the arrows give generators of the module of homomorphisms between these vector bundles.

In this paper we introduce the \emph{Lusztig algebra}  $\widetilde A(G)$ for any finite subgroup $\sfG \subseteq \GL(n,K)$ and a $\sfG-K$-algebra $A$, see Def.~\ref{Def:LusztigAlg}. The Lusztig algebra $\widetilde A(G)$ is Morita equivalent to $A*\sfG$ (Cor.~\ref{Cor:LusztigMoritaequiv}) and  we obtain an embedding of $\widetilde A(G)$  into a matrix algebra over $A$. Thus we get a more conceptual way to study skew group rings $A*\sfG$ and their Morita equivalent basic algebras.  
In particular, if we assume that $A$ is Koszul, one can easily read off the McKay quiver of $\sfG$ and also the relations for the basic Morita equivalent path algebra, see Theorem \ref{Thm:LusztigAlgKoszul}. 
This gives a method to find the quiver and relations for $\widetilde A(\sfG)$. In the case where $A=S$, there is a different method from ours (using super potentials) to obtain the quiver and relations for the Morita equivalent path algebra of $S*G$ in \cite{BocklandtSchedlerWemyss}.

Furthermore, we compute examples of  Lusztig algebras $\widetilde A(\sfG)$ for Koszul algebras $A$ and finite groups $\sfG\subseteq \GL(n,K)$: we compute the Lusztig algebra for $A=S$ and its Koszul dual $A^{!}$ for the group $\sfG=D_4 \subseteq \GL(2,K)$ viewed as a reflection group in Example \ref{Ex:D4}; for $A=S=\Sym_{K}(V)$ and $\sfG$ an abelian group we recover the description of $S*\sfG$ of \cite{Crawetc}, and also of \cite[Cor.~4.1]{BocklandtSchedlerWemyss} in Example \ref{Ex:abelian}; for $\sfG=S_3$ acting on a vector space of dimension $2$ we give a general description of $\widetilde A(G)$ as a matrix algebra over $A$ (Thm.~\ref{Thm:LuzstigAlgS3}), and then calculate the relations in $\widetilde A (S_3)$ for all algebras $A$ that are free algebras modulo a $\sfG$-stable quadratic ideal. In the last example we compute the quiver and relations for $\widetilde S(\sfG)$, where $\sfG=G(1,1,4) \cong S_4$ with the help of Macaulay2, see Example \ref{Ex:S4}.

Finally a comment on how this paper came into being: 
In the first preprint version on the arXiv of \cite{BFI} we had already determined the McKay quivers of $G(1,1,n)$, and Section \ref{Sec:Sn} of the current paper appeared originally there. However, it made sense to combine it with the results of M.~Lewis' thesis about the McKay quivers of the general $G(r,p,n)$ \cite{LewisThesis}, that now make up Sections \ref{Sec:Gm1n} and \ref{Sec:Grpn} of the current paper. 
The second part of this paper about Lusztig algebras is an attempt to find a conceptual way to express the basic versions of skew group rings $S*\sfG$, and was suggested by the first author (first discussed in \cite{RagnarKLtalk}). It would be interesting to explore these algebras and their role in the McKay correspondence further.

\subsection{Plan of the paper} After discussion of some preliminaries in Section \ref{Sec:prelim} we compute the McKay quivers for the symmetric groups $G(1,1,n)$ in terms of Young diagrams in Section \ref{Sec:Sn}. In Section \ref{Sub:Gr1nIrreps} first the description of the irreducible representations of $G(r,1,n)$ for $r >1$ is recalled, as well as Ariki--Koike's branching rule. Then we determine restriction and induction of $G(r,1,n)$ and $G(r,1,n-1)\times \mu_r$ (see Prop.~\ref{Prop:IndRes}), which leads to the main result of this section, the description of the McKay quiver of $G(r,1,n)$ in Theorem \ref{thm:ML}. In Section \ref{Sub:IrrepGrpn} the irreducible representations of $G(r,p,n)$ for $p >1$ are described and Stembridge's induction and restriction rules from $G(r,1,n)$ to $G(r,p,n)$ are recalled. The main result of this Section is the description of the McKay quiver of $G(r,p,n)$ (Theorem \ref{Thm:McKayGrpn}), followed by several examples, in particular the McKay quivers for the Coxeter groups $I_2(p)=G(p,p,2)$ are determined (see Example \ref{Ex:Coxeter2}). 
In the second part, Section \ref{Sec:Lusztig}, we introduce Lusztig algebras: we define $\widetilde{A}(\sfG,W)$ for any $\sfG$-representation $W$, and show that in case $W$ is a representation generator, the Lusztig algebra $\widetilde{A}(\sfG)$ of $\sfG$ on $A$ is Morita equivalent to the skew group ring $A*\sfG$ (Cor.~\ref{Cor:LusztigMoritaequiv}).  In  \ref{Sub:QuiverRelationsKoszul} we discuss how to obtain the quiver and relations of $\widetilde{A}(\sfG)$ in case $A$ is Koszul. The paper ends with several examples in Section \ref{Sub:Examples}: we compute the Lusztig algebras for the group $\sfG=D_4$ viewed as a reflection group in $\GL(2,K)$, and for abelian subgroups in $\GL(n,K)$. In Example \ref{Ex:S3} the reflection group $S_3$ is treated in detail. 
Finally, we compute the relations for $\sfG=S_4$ and $A=K[x_1,x_2,x_3]$.

\subsection{Acknowledgements} The first three authors want to thank the Mathematisches For\-schungszentrum Oberwolfach for the perfect working conditions and inspiring atmosphere. Most of the work for this paper was done in frame of the Leibniz fellowship programme and the Research in Pairs programme. The second author wants to thank Colin Ingalls and Carleton University for their hospitality. We also want to thank the anonymous referee for helpful comments. Last but not least, we thank Ruth Buchweitz for her hospitality and support.
%%%%%%%
\section{Preliminaries} \label{Sec:prelim}

\subsection{Groups and quivers}
We fix a finite group $\sfG$ and an algebraically closed field $K$ with $|\sfG|\in K^{*}$. %We assume that the $|G|$ is invertible in $G$.
With the field understood throughout, unadorned tensor products are taken over $K$.
For a vector space $V$ over $K$, denote $V^*$ its $K$--dual.

Given our assumptions, the group algebra $K\sfG$ is semi--simple,
$K\sfG\cong \prod_{i=1}^{r}\End_{K}(V_{i})$,
where $V_{i}$ runs through representatives of the isomorphism classes of irreducible
(or indecomposable) $\sfG$--representations, equivalently, $K\sfG$--modules.  We will also sometimes denote a representation (i.e. a module $W$ over the group algebra $K\sfG$) as a group homomorphism $\rho: \sfG \xrightarrow{} \GL(W)$.

From the group $\sfG$, one constructs the \emph{McKay graph} or \emph{McKay quiver} $\Gamma(\sfG,R)$ (also called \emph{representation graph}, see \cite{FordMcKay}) as follows: let $R$ be any representation of $\sfG$ and let $\Irrep(\sfG)=\{V_1, \ldots, V_r \}$ be the set of irreducible representations of $\sfG$. The nodes of $\Gamma(\sfG,R)$ are the irreducible representations $V_i$ of $\sfG$ and there are $m_{ij}$ arrows from $V_i$ to $V_j$ if $V_j$ appears with multiplicity $m_{ij}$ in $V_i \otimes R$. In other words: $m_{ij}=\dim_K(\Hom_{K\sfG}(R \otimes V_i,V_j))$.

Note that $\Gamma(\sfG,R)$ is a directed graph that may have loops and multiple edges. Most of the time $R$ will be the \emph{defining} or \emph{standard representation} $V=V_{\stn}$ of $\sfG$: with this we mean the the natural embedding $\sfG \hookrightarrow \GL(V)$, where $V$ is an $n$-dimensional $K$-vector space. In this case we will thus only write $\Gamma(\sfG)$ for the McKay quiver. 

\begin{Bem}
For $R$ one may choose any faithful $\sfG$-representation such that $\Gamma(\sfG,R)$ is a connected graph. As Ford and McKay \cite{FordMcKay} remark, $\Gamma(\sfG,R)$ is undirected (with possible loops) if $R$ is a self-dual representation. See Section \ref{Sec:Sn} for the example of $\sfG=S_n$, for which $\Gamma(\sfG)$ is undirected and has loops for $n \geq 3$.
\end{Bem}

\begin{Bem} 
Let us note that the McKay quiver $\Gamma(\sfG)$ is sometimes defined to be the opposite quiver $\Gamma(\sfG)^{op}$, i.e., there are $m_{ij}$ arrows from $V_i$ to $V_j$ if $V_i$ appears with multiplicity $m_{ij}$ in $V_j \otimes V$. This is in particular in the literature on representation theory, see e.g.~\cite{Auslander86, AuslanderReiten, Yoshino, LeuschkeWiegand, BocklandtSchedlerWemyss}, because $\Gamma(\sfG)^{op}$ for $\sfG \subseteq \mathrm{GL}(2,K)$ then coincides with the Auslander--Reiten quiver of $\sfG$. Our convention is used in the geometric context, see \cite{FordMcKay, GonzalesSprinbergVerdier, BFI}.
\end{Bem}

\subsection{Characters and Frobenius reciprocity} 
The notation is mostly from \cite{JamesLiebeck}, but also see \cite{Serre, FultonHarris} for more background. For any representation $W$ of $\sfG$ we denote its character $\chi_W$, and the characters associated to the irreducible representations $V_i$ by $\chi_{V_i}$. Then any character $\chi$ can be written $\chi=\sum_{i=1}^r n_i \chi_{V_i}$ for some $n_i  \geq 0$. The $\chi_{V_i}$ with $n_i >0$ are called the constituents of $\chi$. 
We further denote the inner product of two characters $\chi,\varphi$ by
$$\langle\chi, \varphi\rangle=\frac{1}{|\sfG|}\sum_{i=1}^r|g_i^\sfG|\chi(g_i)\overline{\varphi(g_i)} \ , $$
where $g_1, \ldots, g_r$ are representatives of the conjugacy classes $g_i^\sfG$ of $\sfG$ (see \cite[Chapter 14]{JamesLiebeck}). We will need that for any character $\chi$ we have
$$\chi=\sum_{i=1}^ra_i \chi_{V_i} \ , \quad \textrm{ with } a_i=\langle\chi,\chi_{V_i}\rangle \ .$$
\begin{Bem}
Since we assume $K$ to be algebraically closed, the characters determine the irreducible representations and for the McKay quiver $\Gamma(\sfG)$ we get that (see e.g., \cite[Cor.~2.16]{FultonHarris})
$$m_{ij}=\langle \chi_{V\otimes V_i}, \chi_{V_j} \rangle = \langle \chi_V \cdot \chi_{V_i}, \chi_{V_j} \rangle \ . $$ Thus, in order to compute the arrows in the McKay quiver of $\sfG$, it is enough if we know the character table of $\sfG$.
\end{Bem}
For a subgroup $\sfH$ of $\sfG$, we denote by $\Res^\sfG_\sfH(W)$ (resp. $\Res^\sfG_\sfH(\chi)$) the restriction of the representation $W$ (resp. the character $\chi$) to $\sfH$. Further, for a representation $W'$ of $\sfH$ with character $\chi'$ we denote the induced representation to $\sfG$ by $\Ind^\sfG_\sfH(W')$ and the induced character by $\Ind^\sfG_\sfH(\chi')$. When the groups $\sfG$ and $\sfH$ are clear, we only write $\Res$ and $\Ind$. We will frequently need 

\begin{Thm}[Frobenius reciprocity] \label{Thm:Frobenius}
Let $\sfH$ be a subgroup of $\sfG$ and suppose that $\varphi$ is a class function on $\sfH$ and that $\chi$ is a class function on $\sfG$. Then
$$\langle \varphi, \Res^\sfG_\sfH(\chi)\rangle_\sfH=\langle\Ind^\sfG_\sfH(\varphi), \chi\rangle_\sfG \ . $$
Here the subscript tells us over which group the pairing has to be computed. Note that sometimes we will use this theorem for the associated representations $W'$ of $\varphi$ and $W$ of $\chi$. On this level we have
$$\Ind^\sfG_\sfH(\Res^\sfG_\sfH(W) \otimes W') = W \otimes \Ind^\sfG_\sfH(W') \ , $$
cf.~\cite[3.3, example 5]{Serre}.
\end{Thm}

\subsection{The reflection groups $G(r,p,n)$} \label{Sub:Grpn}

Here we will state some basic properties for $G(r,p,n)$, for ease of reference. For more information about these groups, see \cite[Chapter 2]{LehrerTaylor}, where it is shown (among other facts) that these groups are generated by pseudo-reflections. The irreducible representations of these groups will be treated in Sections \ref{Sec:Gm1n} and \ref{Sec:Grpn}. 

Note that the integers $r,p,n$ have to be positive and satisfy $p | r$ and $(r,p,n) \neq (2,2,2)$, since this is the only case when $G(r,p,n)$ is not irreducible (for definition of irreducibility and a proof of this fact see e.g. \cite{LehrerTaylor}). \\
These groups can be explicitly described as follows: 
$G(r,p,n) \subseteq \GL(V)$, where $V$ is a $K$-vector space of dimension $n$ consists of $n \times n$ matrices $A$ such that: 
\begin{enumerate}
\item the entries $A_{ij}$ are either $0$ or an $r$-th root of unity, 
\item there is exactly one $A_{ij} \neq 0$ in each row and each column, 
\item the $\frac{r}{p}$-th power of a product of the non-zero entries is $1$.
\end{enumerate}
In other words, a matrix $A$ is in $G(r,p,n)$ if it is of the form $PD$, where $P$ is an $n \times n$ permutation matrix and $D$ is an $n \times n$ diagonal matrix whose entries are $r$-th roots of unity and $(\det D)^{\frac{r}{p}}=1$.

\begin{ex} 
\begin{enumerate}[leftmargin=*]
\item The group $G(1,1,n)$ is the symmetric group $S_n$.
\item The group $G(2,1,n)$ is the Coxeter group $B_n$. It is the semi-direct product $\mu_2^n \rtimes S_n$.
\item
The groups $G(r,r,n)$ are true reflection groups, that is, every generating reflection is of order $2$. %They contain $m {n \choose 2}$ reflections, see \cite[Prop.~2.9]{LehrerTaylor}. 
In particular, if $r=2$, then $G(2,2,n)$ is the Coxeter group $D_n$. 
\item The only other true reflection groups in the family $G(r,p,n)$ are the $G(2p,p,n)$, which can be shown by analyzing the possible orders of reflections \cite[Chapter 2]{LehrerTaylor}. %this is an exercise in LT: ex 15 p38
\end{enumerate}
\end{ex}

\begin{lemma} \label{Lem:GrpnNormal}
The groups $G(r,p,n)$ with $p >1$ and $p|r$  are normal subgroups of $G(r,1,n)$ and satisfy the following short exact sequence
$$ 1 \xrightarrow{} G(r,p,n) \xrightarrow{} G(r,1,n) \xrightarrow{\phi} \mu_p \xrightarrow{} 1 \ , $$
where $\phi(M)=\phi(P \cdot D)=(\det(D)^{\frac{r}{p}})$.
\end{lemma}

\begin{proof} Direct calculation.
\end{proof}

\subsection{Young diagrams} \label{Sub:Youngdiag}

We briefly recall here Young diagrams and tableaux and their relation to representations of $S_n$ and the $G(r,p,n)$. See \cite{fulton1997young} for a more in-depth treatment of Young tableaux.

Recall that a partition $\lambda$ of a positive integer $n$ is a weakly decreasing sequence of positive integers whose sum is $n$. That is, $\lambda = \{\lambda_1, \lambda_2, \cdots,\lambda_k \}$, such that $\lambda_i \geq \lambda_{i+1}$ for all $i$ and $\sum_i \lambda_i= n$.
A partition can be represented visually by a \emph{Young diagram}:
Given a partition $\lambda$ of $n$, the Young diagram associated to $\lambda$ is defined to be a collection of cells arranged in left-justified rows whose lengths are the elements of the partition $\lambda$. These partitions uniquely define the Young diagram. Thus, where no confusion is possible, we will simply refer to $\lambda$ as the Young diagram or partition interchangeably.

\begin{ex}
 Let $\lambda = (3,3,2)$ be a partition of $8$, then the corresponding Young diagram is $\ydiagram{3,3,2}$.
\end{ex}

A \emph{Young tableau} is a Young diagram, of a given size $n$, that has its cells populated with numbers between $1$ and $n$, subject to the following constraints: numbers that appear in the same row, must be \emph{weakly} increasing and numbers in the same column must be \emph{strictly} increasing. A \emph{standard} Young tableau is a Young tableau that uses the numbers between $1$ and $n$ exactly once each.

\begin{ex}
The Young diagram $\ydiagram{2,2}$ of size 4 has only two possible standard tableaux: $\tiny \young(12,34)$ and $\tiny \young(13,24)$. 
\end{ex}
It is a well-known result that the irreducible representations of $S_n$ are in bijection to Young diagrams of size $n$, see e.g., \cite[Prop.~7.2.1]{fulton1997young}. We will usually write $V_\lambda$ when we mean the representation of $S_n$ corresponding to the partition $\lambda$ of $n$ and we will also use $\lambda$ to denote the corresponding Young diagram.

The dimension of each irreducible representation is given by the number of possible standard tableaux on its associated Young diagram. One enumerates the cells in a given Young diagram $\lambda$ by $(i,j)$, where $i$ denotes the row and $j$ the column. Note that one can compute the dimension of $V_\lambda$ with the hooklength formula, see \cite[4.12]{FultonHarris}. 

The representations of $G(r,p,n)$ can be described by $r$-tuples of Young diagrams. This is explained in Sections \ref{Sub:Gr1nIrreps} and \ref{Sub:IrrepGrpn}.

\section{McKay quivers of the symmetric groups $G(1,1,n)$} \label{Sec:Sn}

We describe the McKay quiver of the group $S_n$ with its standard irreducible representation as a reflection group $V_{\stn}$.  
As discussed in Section \ref{Sub:Youngdiag} the vertices of the McKay quiver are given by partitions $\lambda$ of $n$ (which we will label by Young diagrams).

\begin{ex}
Let $\lambda=(n-1,1)$ be a partition of $n$. Then one can show (see \cite[exercise 4.6]{FultonHarris}) that $V_{\lambda}=V_{\stn}=V$ and moreover that the partition (a \emph{hook})
$$\lambda=(n-s,\underbrace{1, \ldots, 1}_{s})$$
corresponds to the irreducible representation $\bigwedge^sV$. In particular, for $s=0$ we see that $V_{(n)}=V_{\triv}$.
\end{ex}

Consider $S_{n-1}$ as a subgroup of $S_n$ by taking the permutations 
that fix $n$, and let 
\begin{center}
\begin{tabular}{lll}
$\Res:$ & $\Mod  KS_n$ & $\longrightarrow \Mod KS_{n-1}$ \\
$\Ind:$ &  $\Mod  KS_{n-1}$&  $\longrightarrow \Mod KS_{n}$
\end{tabular}
\end{center}
be the restriction and induction functors.  Frobenius reciprocity (Thm.~\ref{Thm:Frobenius}) states that these functors are adjoint. In order to describe the arrows in the McKay quiver, we need the following lemmas.

\begin{lemma}[\cite{FultonHarris},Ex.~4.43] Let $\lambda$ be a partition of $n$.  Then
  $$  \Res V_{\lambda} = \bigoplus_{\tau} V_\tau \, ,$$
  where the sum is over all partitions $\tau$ of $n-1$ that are obtained by removing a single block from $\lambda$.
\end{lemma}
\begin{lemma}[\cite{FultonHarris},Ex.~3.16,Ex.~3.13] 
Write $V_{\trv}$ for the trivial representation. For any representation $W$ of $S_n$, if we induce and restrict to a general subgroup, we get
\begin{align*} \Ind \Res W &\cong  W \otimes \Ind V_{\trv} \\
\intertext{and for the particular case of $S_{n-1} \subset S_n$ we have}
\Ind V_{\trv} &\cong  V_{\stn} \oplus V_{\trv} \, .
\end{align*}
\end{lemma}
We write the number of distinct parts of a partition $\lambda$ %=(\lambda_1,\ldots,\lambda_k)$
 as $p(\lambda)$.

\begin{Thm} \label{McKayquiverSn}
Consider the McKay quiver $\Gamma(S_n,\stn)$ of $S_n$ with vertices indexed by partitions of $n$.
Let $\tau \neq \lambda$ be partitions of $n$.
Then there is an arrow from $\lambda$ to $\tau$ if and only if $\lambda$ can be formed from $\tau$ by moving a single block.
The number of loops on $\lambda$ is $p(\lambda)-1,$ one less than the number of distinct parts of $\lambda$.
\end{Thm}
\begin{proof}
Consider two representations $V_\lambda, V_\tau$ of $S_n$.  We have
\begin{eqnarray*}
 \Hom_{KS_{n-1}}(\Res V_\lambda, \Res V_\tau) & \cong & \Hom_{KS_n}(\Ind \Res V_\lambda, V_\tau) \\
& \cong & \Hom_{KS_n}(V_\lambda \otimes \Ind V_{\trv}, V_\tau) \\
& \cong & \Hom_{KS_n}(V_\lambda \otimes (V_{\stn} \oplus V_{\trv}), V_\tau) \\
& \cong &  \Hom_{KS_n}((V_\lambda \otimes V_{\stn}) \oplus V_{\lambda}, V_\tau) \ .
\end{eqnarray*}
Now if $\lambda \neq \tau$ we have that $\Hom_{KS_n}( V_{\lambda}, V_\tau) =0.$  Since $V_{\stn} \cong V_{\stn}^\vee$ 
we obtain  
$$\Hom_{KS_n}(V_\lambda \otimes V_{\stn}, V_\tau) \cong \Hom_{KS_n}(V_\lambda,V_{\stn}\otimes V_\tau) \ , $$ 
whose dimension is the number of arrows from $\lambda$ to $\tau$ in the McKay quiver.
So we only need to note that $\dim \Hom_{KS_{n-1}}(\Res V_\lambda, \Res V_\tau)$ is also the number of isomorphic irreducible summands of 
$\Res V_\lambda$ and  $\Res V_\tau$.  Since we obtain these from removing single blocks from $\lambda$ and $\tau$ there can be at most one 
in common and we obtain the description as stated in the theorem.
Now suppose that $\lambda=\tau$ so we obtain 
\begin{eqnarray*}
\Hom_{KS_{n-1}}(\Res V_\lambda, \Res V_\lambda) & = & \Hom_{KS_n}((V_\lambda \otimes V_{\stn}) \oplus V_{\lambda}, V_\lambda)\\
& = &  \Hom_{KS_n}(V_\lambda \otimes V_{\stn}, V_\lambda) \oplus K\\
& \cong &  \Hom_{KS_n}(V_\lambda, V_{\stn}\otimes V_\lambda) \oplus K
\end{eqnarray*}
Now note that $\dim \Hom_{KS_{n-1}}(\Res V_\lambda, \Res V_\lambda)-1$ is the number of loops on $\lambda$ and 
 this is also the number of partitions $\tau$ of $n-1$ obtained from $\lambda$ by removing one block.  Note that this number is $p(\lambda)-1$.
\end{proof}

\begin{example}  \label{Ex:S4quiver}
Let $G=G(1,1,4)$. Here $G=S_4$ and we have $5$ irreducible representations, labelled by $K_{\triv}=\ydiagram{4}$, $V=\ydiagram{3,1}$, $W=\ydiagram{2,2}$, $V'=\ydiagram{2,1,1}$, $K_{\det}=\ydiagram{1,1,1,1}$. Then the recipe from Theorem \ref{McKayquiverSn} yields the McKay quiver $\Gamma(S_4)$ below in  \eqref{Eq:S4-McKayQuiver}. Note that there are loops in the McKay quiver appearing. By the theorem, for $n\geq 6$ there will be vertices with $\geq 2$ loops.
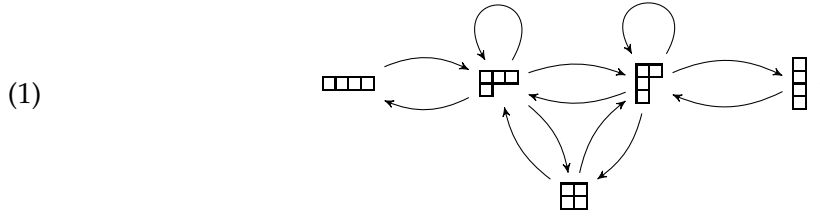
\begin{equation} \label{Eq:S4-McKayQuiver}
{\begin{tikzpicture}[baseline=(current  bounding  box.center)] 
\ytableausetup{centertableaux,boxsize=0.4em}
\node (triv) at (0,0) {$\ydiagram{4}$};

\node (R1) at (2,0) {$\ydiagram{3,1}$} ;
\node (R2) at (4,0)  {$\ydiagram{2,1,1}$};
\node (det) at (6,0)  {$\ydiagram{1,1,1,1}$};
\node (W) at (3,-1.5)  {$\ydiagram{2,2}$};

\draw [->,bend left=25,looseness=1,pos=0.5] (triv) to node[]  [above]{$ $} (R1);
\draw [->,bend left=25,looseness=1,pos=0.5] (R1) to node[]  [above]{$ $} (triv);

\draw [->,bend left=20,looseness=1,pos=0.5] (R1) to node[] [below] {$ $} (R2);
\draw [->,bend left=20,looseness=1,pos=0.5] (R2) to node[] [below] {$ $} (R1);

\draw [->,bend left=25,looseness=1,pos=0.5] (R2) to node[] [below] {$ $} (det);
\draw [->,bend left=25,looseness=1,pos=0.5] (det) to node[] [below] {$ $} (R2);

\draw [->,bend left=20,looseness=1,pos=0.5] (R1) to node[] [below] {$ $} (W);
\draw [->,bend left=20,looseness=1,pos=0.5] (W) to node[] [below] {$ $} (R1);
\draw [->,bend left=20,looseness=1,pos=0.5] (R2) to node[] [below] {$ $} (W);
\draw [->,bend left=20,looseness=1,pos=0.5] (W) to node[] [below] {$ $} (R2);

\path[->,every loop/.style={looseness=5.8}] (R2)
         edge  [in=120,out=60,loop]  ();
         
 \path[->,every loop/.style={looseness=8}] (R1)
         edge  [in=120,out=60,loop]  ();

\node (DD) at (6.3,0) {.};
\end{tikzpicture}} 
\end{equation}
This quiver appeared in \cite[Section 6]{BFI}, but there was no computation provided.

\end{example}

\section{McKay quivers of $G(r,1,n)$} \label{Sec:Gm1n}

\subsection{Irreducible representations of $G(r,1,n)$} \label{Sub:Gr1nIrreps}
 In \cite{ariki1994hecke} generators and relations for the reflection groups $G(r,1,n)$ were determined, as well as a combinatorial description of the irreducible representations of $G(r,1,n)$ in terms of partitions of $n$ (resp.~Young diagrams): every irreducible representation $V_\lambda$ corresponds to an $r$-tuple $\lambda=(\lambda^{(1)}, \ldots, \lambda^{(r)})$, where $\lambda^{(i)}$ is a partition of $n_i$ for $0 \leq n_i \leq n$ and $i=1, \ldots, r$ such that $\sum_{i=1}^rn_i=n$. Note here that $\lambda^{(i)}$ can be empty (the partition of $0$). Thus we can write $\lambda$ as an $r$-tuple of Young diagrams, where we denote the empty set with $-$. The vector space $V_\lambda$ is the formal linear combination of all possible standard Young tableaux of shape $\lambda$. We will denote such $r$-tuple of standard Young tableaux by bold ${\bf t}$.

\begin{example}
Let $\lambda = \left(\ydiagram{2,1}, \ydiagram{1}, - \right)$ be a triple of Young diagrams of size $4$. Then $V_\lambda$, the corresponding irreducible representation of $G(3,1,4)$, is the span of ${\bf t}_i$, $i=1, \ldots, 8$:
\begin{multline*}
\left\{ \left( \tiny \young(12,3)\ ,\young(4) \ , - \right) \ , \left( \tiny\young(13,2)\ ,\young(4) \ , - \right), \left( \tiny\young(12,4) \ ,\young(3) \ , - \right), \left( \tiny \young(14,2) \ ,\young(3)\ , - \right), \right . \\
	 \left .\left( \tiny \young(13,4) \ ,\young(2)\ , - \right), \left( \tiny \young(14,3) \ ,\young(2) \ , - \right) , \left( \tiny \young(23,4) \ ,\young(1) \ , - \right), \left( \tiny \young(24,3) \ ,\young(1) \ , - \right) \right\},
\end{multline*}
showing that $V_\lambda$ is $8$-dimensional.
\end{example}
Note that the dimension of $V_\lambda$ can also be computed with a hook length formula, see Thm.~3.1.8~\cite{LewisThesis}.

The next step is to identify the defining and the trivial representation of $G(r,1,n)$. For this, we have to identify the action of the generators of $G(r,1,n)$ on the $r$-tuples of standard Young tableaux ${\bf t}$.   
In \cite[Cor.~3.14]{ariki1994hecke}, the action of the generators $s_k$ on the $r$-tuples $\mathbf{t}$ is explicitly demonstrated. Since we make use of this action we reproduce it here:

Let $\mathbf{t}_p$ be such an $r$-tuple of Young tableaux. If one can exchange the cells labelled $k$ and $k-1$ and still get a valid standard tableaux, then denote $\mathbf{t}_q$ be the same $r$-tuple with the same diagram, but with the cells labelled $k$ and $k-1$ exchanged. If we fix $\xi$, a primitive $r^{th}-$root of unity, then
 \[ s_1 \mathbf{t}_p = \xi^{i-1}\mathbf{t}_p \textnormal , \]
where the digit '$1$' appears in the $i^{th}$ position of the $r-$tuple. For $k>1$ we have the cases:
\[
s_k \mathbf{t}_p = \begin{cases} \mathbf{t}_p & \textnormal{if $k$ and $k-1$ are in the same row} \\
								-\mathbf{t}_p &\textnormal{if $k$ and $k-1$ are in the same column} \\
								\mathbf{t}_q & \textnormal{otherwise}
							\end{cases} 
\]

Moreover, later (see proof of Prop.~\ref{Prop:IndRes}) we will need the elements $t_k$ of $G(r,1,n)$, which are constructed iteratively by $t_1:=s_1$ and 
$$t_k= s_k \cdot t_{k-1} \cdot s_k \ ,  \text{ for } 1<k \leq n \ .$$
 By \cite[Prop.~3.16]{ariki1994hecke} these elements act on the $r$-tuples $\mathbf{t}_p$ by:
\[ t_k \mathbf{t}_p = \xi^{i-1} \mathbf{t}_p \textnormal , \]
where the digit '$k$' appears in the $i^{th}$ position of the $r-$tuple.

\begin{lemma} Let $V_\tau$ be the space spanned by formal linear combinations of standard Young tableaux on the following Young diagram:
\[ \tau = \left( \ydiagram{2} \cdots \ydiagram{1}, -, \cdots, - \right) \textnormal . \]
Then $V_\tau$ is the trivial representation of $G(r,1,n)$.
\end{lemma}

\begin{proof}
It is clear that $V_\tau$ is one-dimensional: there is only one possible standard tableau on $\tau$:
\[ \mathbf t = \left( \tiny \young(12) \cdots \young(n) \ , - , \cdots,  - \right) \textnormal . \]
It is easy to see that the generators $s_k$ act trivially on $\mathbf t$. Hence $V_\tau $ is the trivial representation of $G(r,1,n)$. 
\end{proof}

\begin{lemma} 
Let $V_\alpha$ be the space spanned by the $r$-tuples of standard Young tableaux with the following Young diagram:
\[\alpha =  \left( \ydiagram{2} \cdots \ydiagram{1} , \ydiagram{1}, -, \cdots, -\right)\textnormal .\]

Let the $r$-tuple $\mathbf{t}_i$ represent the tableau with the numeral $i$ appearing in the single cell of the second entry of the $r-$tuple.
Then we have that $V_\alpha = \mathrm{span} \{ \mathbf{t}_1, \mathbf{t}_2, \cdots , \mathbf{t}_n \}$.
Then the irreducible representation $V_\alpha=V_{\stn}$ is the standard representation of $G(r,1,n)$.
\end{lemma}

\begin{proof}
Note that the choice of $i$ in the second component of $\mathbf{t}_i$ uniquely defines the $n-1$ cells of the first component, so it  is obvious that $\dim V_\alpha = n$. 

From here, we can explicitly construct the matrices associated to the generators of $G(r,1,n)$. For $s_1$ we have that:
\[
s_1: \mathbf{t}_i \longmapsto \begin{cases} \xi \cdot \mathbf{t}_i \textrm{ for $i=1$ , } \\
1 \cdot \mathbf{t}_i \textrm{ for $i > 1$} \ .
\end{cases}
\]
So the matrix associated to $s_1$ in this basis of $V_\alpha$ is the diagonal matrix with $\xi$ in the first position and $1$'s elsewhere. For the other generators $s_k$, with $k > 1$
a straightforward computation shows that the associated matrix on the basis $\mathbf{t}_1, \mathbf{t}_2, \cdots , \mathbf{t}_n$ of $V_\alpha$ is the permutation matrix that interchanges the coordinates $k$ and $k-1$. 
One sees that these matrices are exactly the generators for the matrix presentation of $G(r,1,n)$, which proves that $V_\alpha = V_{\stn}$.
\end{proof}

\subsection{Induction and Restriction} In order to describe the McKay quiver $\Gamma(G(r,1,n),V_{\stn})$, one has to compute tensor products of the irreducible representations with $V_{\stn}$. We will give a combinatorial description of the arrows in the McKay quiver of $G(r,1,n)$ in terms of moving boxes of the $r$-tuples of Young diagrams that correspond to the irreducible representations of the group. Therefore we make use of combinatorial induction and restriction of representations and Ariki--Koike's branching rule.

\begin{defi} Let $\alpha=(\alpha^{(1)}, \ldots, \alpha^{(r)})$ be a representation of $G(r,1,n)$. We say that a cell in the Young diagram $\alpha^{(i)}$ is \emph{available}, if that cell could be deleted resulting in a legal Young diagram.  If $\beta$ is another representation of $G(r,1,m)$, with $m \leq n$, then $\beta \subset \alpha$ means that the $r-$tuple $\beta$ can be obtained from $\alpha$ by merely deleting cells from the Young diagrams $\alpha^{(i)}$. Moreover, we write $\alpha \setminus \beta$ for the  $r$-tuple of (skew) Young diagram resulting from the removal of all of the cells in $\beta$ from $\alpha$. Note that this is defined only if $\beta \subset \alpha$.
\end{defi}

 \begin{prop}[Branching rule, \cite{ariki1994hecke}, Cor.~3.12] \label{Prop:branching}
Let $\alpha$ be an irreducible representation of $G(r,1,n)$. The restriction functor from representations of $G(r,1,n)$ to representations of $G(r,1,n-1)$ is given by:
\[ V_\alpha = \bigoplus \limits_{\substack{\beta \subset \alpha \\ |\alpha \setminus \beta| =1}} V_\beta \textnormal{,}\]
where $\beta$ runs over all possible $r-$tuples obtained by deleting a single available cell of $\alpha$.
\end{prop}

 Restriction does not change the dimension of a representation. Induction, on the other hand, will increases the dimension of a representation by a factor equal to the index of the subgroup. It is easily seen that induction $G(r,1,n-1)$ to $G(r,1,n)$ will multiply the dimension of a representation by $n\cdot r$, while tensoring with $V_{\stn}$ should only multiply the dimension by $n$ (the dimension of $V_{\stn}$). So we need a \emph{larger} subgroup to restrict to. We will see that $G(r,1,n-1) \times \mu_r \hookrightarrow G(r,1,n)$ is the right subgroup.
 
First we describe the group $G(r,1,n-1) \times \mu_r$ and its irreducible representations, where we assume that $\mu_r=\langle \xi \rangle$: an element in this group is of the form $(PD,\xi^i)$, where $PD$ is an element of $G(r,1,n-1)$ as described in Section \ref{Sub:Grpn}. The embedding $G(r,1,n-1) \times \mu_r \hookrightarrow G(r,1,n)$ is then given by
$$(PD, \xi^i) \mapsto \begin{pmatrix} P & 0 \\ 0 & 1 \end{pmatrix} \cdot \begin{pmatrix} D & 0 \\ 0 & \xi^i \end{pmatrix} \ .$$
The character table of $G(r,1,n-1) \times \mu_r$ is the Kronecker product of the respective character tables of $G(r,1,n-1)$ and $\mu_r$, see e.g., \cite[Theorem 19.18]{JamesLiebeck}. Thus it is easy to see that an irreducible representation of $G(r,1,n-1) \times \mu_r$ is a pair of an irreducible representation $\alpha$ of $G(r,1,n-1)$ and a power $\xi^i$ for $0 \leq i \leq r-1$. We denote these pairs by $\alpha  \boxtimes i$.

\begin{lemma} \label{lem:Gr1nxmu}
Given two irreducible representations, $\alpha \boxtimes i$ and $\beta \boxtimes j$ of $G(r,1,n-1)\times \mu_r$, their tensor product is 
\[ (\alpha \boxtimes i) \otimes (\beta \boxtimes j) = (\alpha \otimes \beta) \boxtimes (i \otimes j) = (\alpha \otimes \beta) \boxtimes k \textnormal{,}\]
where $k = i+j \mod r$.
\end{lemma}
\begin{proof}
This is clear by looking at the character table of $G(r,1,n-1)\times\mu_r$.
\end{proof} 

Define $\nu: \Irrep(G(r,1,1)) \xrightarrow{} \NN$ via
$$\nu(\lambda)=\nu((-, \ldots, -, \underbrace{\ydiagram{1}}_{i},-, \ldots, -))=i \ ,$$
that is, $\nu$ returns the position of the single cell in the $r$-tuple $\lambda$.

\begin{proposition} \label{Prop:IndRes}
The restriction functor $\mathrm{Res}^{G(r,1,n)}_{G(r,1,n-1)\times \mu_r}$ is given by:
\[ \alpha \longmapsto \sum \limits_{\substack{\beta \subset \alpha \\ |\alpha \setminus \beta| =1}} \beta \boxtimes ( \nu(\alpha \setminus \beta)-1) \textnormal{,}\]
where $\beta$ runs over all possible ways of (legally) deleting a single cell from $\alpha$. 

The induction functor $\mathrm{Ind}^{G(r,1,n)}_{G(r,1,n-1)\times \mu_r}$ for $\beta=(\beta^{(1)}, \ldots, \beta^{(i)}, \ldots, \beta^{(r)})$ is given by
\[ \beta \boxtimes i \longmapsto \bigoplus_{j}(\beta^{(1)}, \ldots,\beta^{(i)}, \beta^{(i+1)} \cup \{\ydiagram{1}\}, \beta^{(i+2)}, \ldots, \beta^{(r)}) \ , \]
where $\beta^{(i+1)}\cup \{\ydiagram{1}\}$ is a legal Young diagram of size $n_{i+1}+1$ obtained by adding a cell to $\beta^{(i+1)}$ and $j$ runs over all legal Young diagrams that can be obtained this way. Note that we consider $i=0, \ldots, r-1$.

\end{proposition}

\begin{proof}
The restriction $\beta$ from $\alpha$ follows from the branching rule Prop.~\ref{Prop:branching}, so we only need to understand the integer $\nu(\alpha \setminus \beta)-1$. 

Recall the vectors $\mathbf{t}$, the young tableaux of shape $\alpha$ which span $V_\alpha$. When we are deleting a cell from $\alpha$ we are explicitly deleting the cell from $\mathbf t$ which is labelled with $n$. Recall the action $t_n$ from the presentation of $G(r,1,n)$ (see Section \ref{Sub:Gr1nIrreps}):
\[ t_n \cdot \mathbf t \longmapsto \xi^{i-1} \mathbf t \textnormal , \]
where $n$ appears in the $i^{th}$ component of $\mathbf t$. This shows the correspondence.
The induction rule also follows from the definition. 
\end{proof}

What we have shown is that the restriction to this subgroup consists of the deletion of an available cell together with an integer which serves to remember from whence we deleted the cell.  Furthermore, induction from this subgroup consists of the sum of all of the (legal) ways to add a cell to the component of the $r$-tuple corresponding to that power of $\xi$.

\begin{example}
Given the representation $\alpha = \left( \ydiagram{2,1} , \ydiagram{2,1,1}, \ydiagram{1} \right)$ in $G(3,1,8)$ we have: 
\begin{align*}
\mathrm{Res}^{G(3,1,8)}_{G(3,1,7)\times \mu_3} \left(\ydiagram{2,1} , \ydiagram{2,1,1}, \ydiagram{1} \right)= &\left(\left( \ydiagram{1,1} , \ydiagram{2,1,1}, \ydiagram{1}\right)\boxtimes 0 \right)
	\oplus \left( \left( \ydiagram{2} , \ydiagram{2,1,1}, \ydiagram{1} \right)\boxtimes 0 \right)
	\oplus \left( \left(\ydiagram{2,1} , \ydiagram{2,1}, \ydiagram{1} \right)\boxtimes 1 \right) \\
	 & \oplus \left( \left( \ydiagram{2,1} , \ydiagram{1,1,1}, \ydiagram{1} \right)\boxtimes 1 \right)
	\oplus \left( \left( \ydiagram{2,1} , \ydiagram{2,1,1}, - \right)\boxtimes 2 \right)
\end{align*}
\end{example}

\begin{example}
Taking $\left( \ydiagram{2,1} , \ydiagram{2,1}, \ydiagram{1} \right)\boxtimes 1$ and $\left(\ydiagram{2,1} , \ydiagram{2,1,1}, - \right)\boxtimes 2 $ from the previous example we have that:
\[ \mathrm{Ind}^{G(3,1,8)}_{G(3,1,7)\times \mu_3}\left(\left( \ydiagram{2,1} , \ydiagram{2,1}, \ydiagram{1} \right)\boxtimes 1 \right) =\left( \ydiagram{2,1} , \ydiagram{2,1,1}, \ydiagram{1} \right)\oplus \left( \ydiagram{2,1} , \ydiagram{2,2}, \ydiagram{1} \right) \oplus \left( \ydiagram{2,1} , \ydiagram{3,1}, \ydiagram{1} \right) \textnormal ,\]
since there are three legal ways of adding a cell to the second component. On the other hand,
\[\mathrm{Ind}^{G(3,1,8)}_{G(3,1,7)\times \mu_3}\left( \left( \ydiagram{2,1} , \ydiagram{2,1,1}, - \right)\boxtimes 2 \right) = \left( \ydiagram{2,1} , \ydiagram{2,1,1}, \ydiagram{1} \right) \textnormal ,\]
since there is only one way to add a cell to the third component.
\end{example}

\subsection{The quiver} We are ready to describe the McKay quiver of $G(r,1,n)$. 
Note that by classical results of Burnside--Brauer \cite{Brauer} $\Gamma(G(r,1,n), \stn)$ is a connected graph (since $V_{\stn}$ is faithful).

\begin{Thm} \label{thm:ML} 
  Let $\Gamma=\Gamma(G(r,1,n),\stn)$ be the McKay graph of $G(r,1,n)$. Then the vertices of $\Gamma$ correspond to $r$-tuples $\lambda=(\lambda^{(1)}, \ldots, \lambda^{(r)})$ of Young diagrams of size $0 \leq n_i \leq n$, such that $\sum_{i=1}^rn_i=n$.
  
There is an arrow from a vertex $\alpha$ to  vertex $\beta$ in $\Gamma$ if and only if the $r-$tuple of Young diagrams $\beta$ can be obtained from $\alpha$ by deleting a cell from position $i$ in $\alpha$ and then adding a cell to position $i+1 \mod r$.
\end{Thm}

\begin{proof}
  The irreducible representations of $G(r,1,n)$ were described in Section \ref{Sub:Gr1nIrreps}. Let $W = \left( \ydiagram{2} \cdots \ydiagram{1}, -, \cdots ,-\right)\boxtimes1 $ be a representation of $\sfH = G(r,1,n-1)\times \mu_r$ and $V_\alpha$ be an irreducible representation of $\sfG = G(r,1,n)$. Notice that the Young diagram part of $W$ is associated to the trivial representation of $G(r,1,n-1)$.
  
Now it follows that (apply Frobenius reciprocity, Theorem \ref{Thm:Frobenius})\begin{align*}
\mathrm{Ind}^\sfG_{\sfH} \left\{ \mathrm{Res}^\sfG_{\sfH}V_\alpha\otimes W\right\} & =V_\alpha \otimes \mathrm{Ind}^\sfG_\sfH  W 
=V_\alpha \otimes \left(\ydiagram{2}\cdots \ydiagram{1},\ydiagram{1},-, \cdots ,-\right)\\
& = V_\alpha \otimes V_{\stn} \textnormal \ .
\end{align*}
We know from Prop.~\ref{Prop:IndRes} that $\mathrm{Res}^\sfG_{\sfH}V_\alpha$ can be understood as a direct sum of $r$-tuples of Young diagrams, each obtained by deleting exactly one cell from $V_\alpha$ together with a power of $\xi^i$, that records from which component the cell was deleted.

$\mathrm{Res}^\sfG_\sfH V_\alpha\otimes W$ can then be understood by fixing the Young diagrams in $\mathrm{Res}^\sfG_\sfH V_\alpha$ (since the Young diagram in $W$ corresponds to the trivial representation) and increasing the power of $\xi$ in each summand of $\mathrm{Res}^\sfG_\sfH V_\alpha$  by $1$ (modulo $r$).

Inducing on the resulting sum will mean adding a cell in each summand to the component exactly one step to the right (cyclically) of the component from which the cell was originally deleted. 

This is exactly the statement of the theorem.
\end{proof}

We illustrate the theorem with an example.
\begin{example}
 Consider the irreducible representation $\alpha = \left( \ydiagram{3,1}, \ydiagram{1}, -\right)$ of $G(3,1,5)$. The goal is to find all  the targets of arrows with this vertex as a source. Note that $V_{\stn}$ corresponds to $\left( \ydiagram{4}, \ydiagram{1}, - \right)$ in this case.
Letting $\sfG= G(3,1,5)$ and $\sfH = G(3,1,4) \times \mu_3$, we have that:
\begin{align*}
\mathrm{Res}^\sfG_\sfH V_\alpha = \left(\left( \ydiagram{3}, \ydiagram{1}, -\right)\boxtimes 0\right)
	\oplus \left(\left( \ydiagram{2,1}, \ydiagram{1}, -\right)\boxtimes 0\right)
	\oplus \left(\left( \ydiagram{3,1}, -, -\right)\boxtimes 1\right).
\end{align*}
Now, the tensor product with $\left( \left( \ydiagram{4}, -, -\right)\boxtimes 1\right)$ yields:
\begin{align*}
\mathrm{Res}^\sfG_\sfH V_\alpha \otimes\left( \left( \ydiagram{4}, -, -\right)\boxtimes 1\right) = \left(\left( \ydiagram{3}, \ydiagram{1}, -\right)\boxtimes 1\right)
	\oplus \left(\left( \ydiagram{2,1}, \ydiagram{1}, -\right)\boxtimes 1\right)
	\oplus \left(\left( \ydiagram{3,1}, -, -\right)\boxtimes 2\right).
\end{align*}
Then induction yields:
\begin{align*}
\mathrm{Ind}^\sfG_\sfH \left\{\mathrm{Res}^\sfG_\sfH V_\alpha \otimes\left( \left( \ydiagram{4}, -, -\right)\boxtimes 1\right)\right\} =& \left( \ydiagram{3,1}, \ydiagram{1}, -\right) \otimes \left( \ydiagram{4}, \ydiagram{1},-\right) \\
	 = &\left( \ydiagram{3}, \ydiagram{2}, -\right) \oplus \left( \ydiagram{3}, \ydiagram{1,1}, -\right)
	\oplus \left( \ydiagram{2,1}, \ydiagram{2}, -\right)\oplus \\
 \phantom{=} &	 \left( \ydiagram{2,1}, \ydiagram{1,1}, -\right)
	\oplus \left( \ydiagram{3,1}, -, \ydiagram{1}\right).
\end{align*}
In this example, we can clearly see the available cells are each moving exactly one component to the right and being placed in a way that creates a legal Young diagram.
\end{example}

\begin{corollary} \label{Cor:nodoublearrows}
The McKay quiver $\Gamma(G(r,1,n),\stn)$ has arrows in both directions if and only if $r=2$, that is, only if $G(r,1,n)$ is a Coxeter group. Furthermore, for vertices $\alpha \neq \beta$ there is at most one arrow from $\alpha$ to $\beta$. Moreover, there are loops in $G(r,1,n)$ if and only if $r=1$.
\end{corollary}

\begin{proof}
This follows directly from the description of the arrows in the McKay quiver of Theorem \ref{thm:ML}. The existence of loops in $\Gamma(G(1,1,n))$ follows from Theorem \ref{McKayquiverSn}. If $r \geq 2$, then the description of arrows in Theorem \ref{thm:ML} shows that there cannot be any loops, since deleting a cell in component $i$ and adding it to component $i+1$ modulo $r$ of any $\lambda$ will change the shape of $\lambda$.
\end{proof}

\section{McKay quivers of $G(r,p,n)$}  \label{Sec:Grpn}

\subsection{Representations of $G(r,p,n)$} \label{Sub:IrrepGrpn} 
The irreducible representations of $G(r,p,n)$ have been described by Stembridge \cite{stembridge1989eigenvalues}, using Clifford theory. We will mostly follow the notation of \cite{bagno2007colored} (also cf.~\cite{MoritaYamada}), expressing the representations in terms of Young diagrams. The rough idea is that each irreducible representation of $G(r,p,n)$ can be understood as the quotient of an irreducible representation of $G(r,1,n)$ by the actions of  linear representations $\delta_\ell$ of $G(r,1,n)$ that land in the cokernel $\mu_p$ of the group homomorphism $G(r,p,n) \xrightarrow{} G(r,1,n)$, cf.~Lemma \ref{Lem:GrpnNormal}. We first determine this action:
Let $ 0\leq \ell < r-1$ and $\delta_\ell$ be the one-dimensional representation of $G(r,1,n)$ corresponding to the Young diagram:
\[ \delta_\ell = \left( -,-, \cdots ,-, \ydiagram{2} \cdots \ydiagram{1}, -, \cdots, - \right)  ,\]
where the (non-empty) trivial Young diagram appears in the $(\ell+1)^{th}$ position of the $r-$tuple.
Note that $\delta_0$ is the trivial representation of $G(r,1,n)$.

\begin{lemma} \label{Lem:shift}
The linear representations $\delta_\ell$ of $G(r,1,n)$, described above, act on the irreducible representations $\lambda$ of $G(r,1,n)$ by cyclically moving each diagram in the $r$-tuple $\ell$ steps to the right. Explicitly: for $\lambda=(\lambda^{(1)}, \ldots, \lambda^{(r)})$ one has that 
$$\lambda \otimes \delta_\ell=(\lambda^{(1-\ell)}, \lambda^{(2-\ell)} \ldots, \lambda^{(r-\ell)}) \ , $$
where $i-\ell$ is considered modulo $r$.
\end{lemma}
\begin{proof}
This can be proven by analyzing how the group generators act on the representation $\lambda$ and $\lambda \otimes \delta_\ell$: one can explicitly determine this action in terms of Young tableaux, as in Section \ref{Sub:Gr1nIrreps}. We leave the straightforward proof to the reader.
\end{proof}

\begin{example}
Let $\alpha = \left( \ydiagram{2,1}, -, \ydiagram{1,1,1}, \ydiagram{1}\right)$ be a representation of $G(4,1,7)$ then the representation $\delta_2$ are of the form:
\[ \delta_2 = \left( -, -,\ydiagram{7},-\right) \textnormal . \]  
and we have that:
\[ \alpha \otimes \delta_2 = \left( \ydiagram{1,1,1}, \ydiagram{1}, \ydiagram{2,1}, -\right) \textnormal .\]
\end{example}

 The group $G(r,p,n)$ for $p >1$ is a subgroup of $G(r,1,n)$. Note that we must have $p|r$. In the following we denote $\sfG:=G(r,1,n)$ and $\sfH:=G(r,p,n)$ and further set $d:=\frac{r}{p}$.

The quotient $\sfG/\sfH$ is isomorphic to the cyclic group $\mu_p=\langle \delta_1^d \rangle$, with $\delta_1$ as defined above. Note that $\delta_1$ is the linear character of $G(r,1,n)$ with $\sfH \subseteq \ker(\delta_1)$ (see \cite[Section 5]{bagno2007colored}). Then $\sfG/\sfH$ acts on the irreducible representations $\lambda$ of $G(r,1,n)$ by $\lambda \mapsto \delta \otimes \lambda$, where $\lambda=(\lambda^{(1)}, \ldots, \lambda^{(r)})$ and $\delta$ is a 1-dimensional representation, as described in Lemma \ref{Lem:shift}. For this, one defines the \emph{shift-operator}
$$\lambda^{\shift{1}}=(\lambda^{(r)},  \lambda^{(1)},\ldots, \lambda^{(r-1)}) \ .$$
We denote by $\lambda^{\shift{i}}$ the irreducible representation of $G(r,1,n)$ we get from applying $i$-times the shift-operator. Then (see e.g., \cite[Section 4]{MoritaYamada}) $\delta_\ell \otimes \lambda = \lambda^{\shift{\ell}}$.

Denote by $[\lambda]$ the $\sfG/\sfH$-orbit of the irreducible $\sfG$-representation $\lambda$. Further define an equivalence relation on $\sfG$-representations via: 
$$\lambda \simeq \mu \textrm{ if }  \mu=\lambda^{\shift{i \cdot d}} \textrm{ for some } i \in \{0, \ldots, p-1\} \ .$$
Then one easily sees that $[\lambda]=\{ \mu : \lambda \simeq \mu \}$. Let $b(\lambda)$ be the cardinality of the $\sfG/\sfH$-orbit of $\lambda$ and $u(\lambda)=\frac{p}{b(\lambda)}$. Furthermore, denote the \emph{stabilizer of $\lambda$}
$$(\sfG/\sfH)_\lambda : =\{ \delta \in \sfG/\sfH: \lambda = \delta \otimes \lambda \} \ . $$
This stabilizer is a subgroup of $\sfG/\sfH$ generated by $\delta_1^{b(\lambda) \cdot d}$ of order $u(\lambda)$. Using Clifford theory, one can prove that the restriction of the irreducible representation $\lambda$ of $\sfG$ restricts to $u(\lambda)=|(\sfG/\sfH)_\lambda|$ irreducible $\sfH$-representations. More precisely, the irreducible representations of $\sfH$ can be classified (originally due to Stembridge):

\begin{prop}(cf.~Stembridge \cite[Prop.~6.1]{stembridge1989eigenvalues}, \cite[Thm.~5.1]{bagno2007colored}) \label{Prop:Stembridge} There is a 1-1 correspondence between the irreducible representations of $\sfH=G(r,p,n)$ and ordered pairs $([\lambda], \delta)$ where $[\lambda]$ is the orbit of the irreducible representation $\lambda$ of $\sfG=G(r,1,n)$, and $\delta \in (\sfG/\sfH)_\lambda$. We have the following restriction and induction rules:
\begin{enumerate}[(a)]
\item \label{Prop:resirr} $\Res^\sfG_\sfH(V_\lambda)=\Res^\sfG_\sfH(V_\mu)$ for any $\mu \simeq \lambda$, 
\item \label{Prop:res} $\Res^\sfG_\sfH(V_\lambda)=\oplus_{\delta \in (\sfG/\sfH)_\lambda}V_{([\lambda], \delta)}$ , 
\item \label{Prop:ind} $\Ind^\sfG_\sfH(V_{([\lambda], \delta)})=\oplus V_\mu$, where the direct sum is over all distinct $\mu \simeq \lambda$.
\end{enumerate}
\end{prop}

\begin{Bem} \label{Rmk:Resirred}
 In particular, from \ref{Prop:resirr} and \ref{Prop:res} it follows that for $\lambda \in \Irrep(\sfG)$, the restriction $\Res^\sfG_\sfH(V_\lambda)$ is irreducible if and only if $|(\sfG/\sfH)_\lambda|=u(\lambda)=1$.  If $u(\lambda) >1$, then the restriction $\Res^\sfG_\sfH(V_\lambda)$ of \ref{Prop:res} will split into $u(\lambda)$ many irreducibles. Moreover, note that in \ref{Prop:ind} , each irreducible $V_\mu$ appears exactly once. 
\end{Bem}

\begin{notation} \label{Not:Grpn} In order to make the notation less clumsy, if $u(\lambda)=1$, then the restriction $([\lambda],\delta)$ stays irreducible, and we will just denote it by $[\lambda]$ (omitting the second component).
  
If $p | r$, then set again $d=\frac{r}{p}$. It is then helpful to visualize the orbit $[\lambda]$ as $p d$-tuples, since the action of $\sfG/\sfH$ will be in $d$-shifts: the first $d$ Young diagrams of $[\lambda]$ form the first $d$-tuple, and so on. For this first partition the $\sfG$-representation $\lambda=(\lambda^{(1)}, \ldots, \lambda^{(r)})$ into $d$-tuples: $\lambda=((\lambda^{(1)}, \ldots, \lambda^{(d)}),(\lambda^{(d+1)}, \ldots, \lambda^{(2d)}), \ldots, (\lambda^{((p-1)d+1)}, \ldots, \lambda^{(pd)}))$. Then 
$$[\lambda]=[((\lambda^{(1)}, \ldots, \lambda^{(d)}),(\lambda^{(d+1)}, \ldots, \lambda^{(2d)}), \ldots, (\lambda^{((p-1)d+1)}, \ldots, \lambda^{(pd)}))] \ . $$
Clearly $(\sfG/\sfH)_\lambda \cong \mu_{u(\lambda)}$. We choose a generator $\xi$ of $\mu_{u(\lambda)}$, then any element in $(\sfG/\sfH)_{u(\lambda)}$ can be uniquely written as $\xi^t$, $t=0, \ldots, u(\lambda)-1$. Instead of $\delta$ we may write $\xi^t$ (or simply $t$).
\end{notation}

\begin{defi} \label{Def:FundDomain} 
With Notation \ref{Not:Grpn} for an irreducible representation $([\lambda],t)$ with $t \in \{0, \ldots,$ $u(\lambda)-1\}$ of $G(r,p,n)$ we call the $(t+1)$-st of the $d$-tuples of $\lambda$ the \emph{fundamental domain} of $([\lambda],t)$.
\end{defi}

This boils down to the following: if $\lambda$ has no ``symmetry'', that is, if $u(\lambda)=1$, then $[\lambda]$ is the fundamental domain. If $\lambda$ has $u(\lambda) >1$ symmetries, then it has $u(\lambda)$ conjugate representations in $\sfH$, which we can distinguish by their fundamental domains.

\begin{Bem}
The fundamental domain will be used to describe the arrows in the McKay quiver of $G(r,p,n)$, in order to make clear, which Young diagrams can be used for moving boxes.  The fundamental domain is highlighting which of the components of the $p d$-tuples are \emph{active}, in the sense, that these components are potentially available for induction and also for the action $-\otimes V_{\stn}$ described in Section \ref{Sub:McKayGrpn}.
\end{Bem}

\begin{Bem} In the thesis \cite{LewisThesis}, the notation of the irreducible representations of $\sfH$ was more visual: one can place the components of $[\lambda]$ around a circle, so that it is easy to see the rotational symmetries among the Young diagrams.
\end{Bem}

\begin{Bem}
One can also determine the dimension of the restricted representations: each direct summand of $\Res^\sfG_\sfH(V_\lambda)$ will have dimension $\frac{\dim V_\lambda}{u(\lambda)}$.
\end{Bem}

We end this section with a few examples. 

\begin{example}[$r=p$ and $u(\alpha)=1$]
Consider the irreducible representations
\[ \alpha = \left( \ydiagram{2,1}, \ydiagram{2}, \ydiagram{3} \right) , \quad
 \beta =  \left(  \ydiagram{2}, \ydiagram{3}, \ydiagram{2,1} \right) \text{  and } \gamma=( \ydiagram{3} , \ydiagram{2}, \ydiagram{2,1} ) \]
of $\sfG=G(3,1,8)$, and let $\sfH=G(3,3,8)$. Then
$$ \Res^\sfG_\sfH(\alpha)=\Res^\sfG_\sfH(\beta)=[\left( \ydiagram{2,1}, \ydiagram{2}, \ydiagram{3} \right)] \ ,$$
but $\Res^\sfG_\sfH(\gamma)\neq \Res^\sfG_\sfH(\alpha)$.
\end{example}

\begin{example}[$p | r$ and $u(\alpha)=1$] \label{Ex:pdividesr} Consider  the irreducible representation
\[ \alpha = \left( \ydiagram{1} , \ydiagram{2}, \ydiagram{3}, \ydiagram{4} \right) ,\]
of $\sfG=G(4,1,10)$. Let $\sfH=G(4,4,10)$ and $\sfH'=G(4,2,10)$. Then $\Res^\sfG_\sfH(\alpha)=[\alpha]$ and 
$$\Res_{\sfH'}^\sfG(\alpha)=[((\ydiagram{1},\ydiagram{2}),(\ydiagram{3},\ydiagram{4}))]=[((\ydiagram{3},\ydiagram{4}),(\ydiagram{1},\ydiagram{2}))] \ ,$$
since $\delta_{\frac{4}{2}}$ shifts any component of $\alpha$ by $2$. Note that $\alpha'=(\ydiagram{3},\ydiagram{4},\ydiagram{1},\ydiagram{2})$ has the same restriction to $\sfH'$.

For the induction of $[\alpha]$ from $\sfH$ back to $\sfG$ we get, using Prop.~\ref{Prop:Stembridge}, that this splits into $b(\alpha)=4$ irreducible direct summands:
$$\Ind^\sfG_\sfH([\alpha])=(\ydiagram{1}, \ydiagram{2}, \ydiagram{3}, \ydiagram{4}) \oplus (\ydiagram{2}, \ydiagram{3}, \ydiagram{4}, \ydiagram{1}) \oplus ( \ydiagram{3}, \ydiagram{4}, \ydiagram{1}, \ydiagram{2}) \oplus (\ydiagram{4}, \ydiagram{1}, \ydiagram{2},  \ydiagram{3}) \ .$$
Moreover, for the induction of $[((\ydiagram{1},\ydiagram{2}),(\ydiagram{3},\ydiagram{4}))]$ from $\sfH'$ to $\sfG$ we similarly get 
$$\Ind^\sfG_{\sfH'}([\alpha])= (\ydiagram{1}, \ydiagram{2}, \ydiagram{3}, \ydiagram{4}) \oplus (\ydiagram{3}, \ydiagram{4}, \ydiagram{1}, \ydiagram{2}) \ . $$
\end{example}

\begin{example}[$u(\alpha)>1$] Take $\alpha = \left(  \ydiagram{2}, \ydiagram{1}, \ydiagram{2}, \ydiagram{1} \right)$ in $\sfG=G(4,1,6)$ and let $\sfH=G(4,4,6)$. Then $u(\alpha)=2$ and the restriction of $\alpha$ to $\sfH$ is the direct sum
$$\Res^\sfG_\sfH(\alpha)=([\left( \ydiagram{2}, \ydiagram{1}, \ydiagram{2}, \ydiagram{1} \right)],0) \oplus ([\left(\ydiagram{2}, \ydiagram{1}, \ydiagram{2}, \ydiagram{1} \right)],1) \ .$$
  The two irreducibles are identical except for their fundamental domains: the fundamental domain of $([\left( \ydiagram{2}, \ydiagram{1}, \ydiagram{2}, \ydiagram{1} \right)],0)$ is $\left(\underline{ \ydiagram{2}, \ydiagram{1}}, \ydiagram{2}, \ydiagram{1} \right)$ and for $([\left( \ydiagram{2}, \ydiagram{1}, \ydiagram{2}, \ydiagram{1} \right)],1)$ we have $\left( \ydiagram{2}, \ydiagram{1}, \underline{\ydiagram{2}, \ydiagram{1} } \right)$.
  
 Going back, we see that 
 $$\Ind^\sfG_\sfH(\left(\underline{ \ydiagram{2}, \ydiagram{1}}, \ydiagram{2}, \ydiagram{1} \right))=\Ind^\sfG_\sfH( \left( \ydiagram{2}, \ydiagram{1}, \underline{\ydiagram{2}, \ydiagram{1} } \right))=   \left( \ydiagram{2}, \ydiagram{1}, \ydiagram{2}, \ydiagram{1} \right) \oplus \left( \ydiagram{1}, \ydiagram{2}, \ydiagram{1}, \ydiagram{2} \right) \ . $$ 
\end{example}

\begin{example}[$u(\alpha)>1$]
Take $\alpha=\left( \ydiagram{1}, \ydiagram{1}, \ydiagram{1} \right)$ in $G(3,1,3)$. Then 
$$\Res^{G(3,1,3)}_{G(3,3,3)}(\alpha)=([\left( \ydiagram{1}, \ydiagram{1}, \ydiagram{1} \right)],0)\oplus ([\left( \ydiagram{1}, \ydiagram{1}, \ydiagram{1} \right)],1) \oplus ([\left( \ydiagram{1}, \ydiagram{1}, \ydiagram{1} \right)],2) \ ,$$
which we can also write (highlighting the fundamental domains) as
$$\Res^{G(3,1,3)}_{G(3,3,3)}(\alpha)=[\left(\underline{\ydiagram{1}}, \ydiagram{1}, \ydiagram{1} \right)] \oplus [\left( \ydiagram{1}, \underline{\ydiagram{1}}, \ydiagram{1} \right)] \oplus [\left( \ydiagram{1}, \ydiagram{1},\underline{\ydiagram{1}} \right)] \ .$$
The induction $\Ind^{G(3,1,3)}_{G(3,3,3)}([\left( \ydiagram{1}, \ydiagram{1}, \ydiagram{1} \right)],i)$ is the irreducible $\alpha$ for $i=0,1,2$.
\end{example}

\subsection{The quiver} \label{Sub:McKayGrpn}

Let $V_{\stn}$ be the standard representation of $G(r,1,n)$.  It is clear that $\Res^{G(r,1,n)}_{G(r,p,n)}(V_{\stn})$ is the standard representation of $G(r,p,n)$.
Restriction to $G(r,p,n)$ will still yield an irreducible representation $[\stn]$ if and only if $(r,p,n) \neq (2,1,2)$ or $(r,p,n) \neq (2,2,2)$. This can easily be seen from Prop.~\ref{Prop:Stembridge}.

\begin{lemma} \label{Lem:ResHnonzero}
Let $V_\lambda$ be an irreducible representation of $\sfG=G(r,1,n)$ with character $\chi_\lambda$ and set $\sfH=G(r,p,n)$ for some $p >1$, and $u(\lambda)=|(\sfG/\sfH)_\lambda|$. Further let $V_\alpha$ be any irreducible $\sfG$-representation with character $\chi_\alpha$ and $V_{\stn}$ be the standard representation of $\sfG$ with character $\chi_{\stn}$.
Then if $u(\lambda) > 1$ and $$\langle \chi_\lambda \cdot \chi_{\stn}, \chi_{\alpha} \rangle \neq 0 \ , $$
that is, if there is an arrow from $\lambda$ to $\alpha$ in $\Gamma(\sfG)$, then $\Res^\sfG_\sfH(\alpha)$ is in $\Irrep(\sfH)$.
\end{lemma}

\begin{proof} Each indecomposable direct summand of $V_\lambda \otimes V_{\stn}$ must have $u(\lambda)=1$, since moving one cell from component $i$ to component $i+1$ (modulo $r$) in $\lambda$ will certainly break the symmetry of $\lambda$. The statement of the lemma now follows from Remark \ref{Rmk:Resirred}.
\end{proof}

\begin{lemma} \label{Lem:symmetryDiv}
Let $V_\lambda$ be an irreducible representation of $\sfG=G(r,1,n)$ with standard representation $V_{\stn}$ and set $\sfH=G(r,p,n)$ for $p >1$, and $u(\lambda)=|(\sfG/\sfH)_\lambda|$ the order of the stabilizer of $\lambda$. Let $V_\varphi$ be any irreducible $\sfH$-representation. Then
$$u(\lambda) | \langle \chi_{\lambda} \cdot \chi_{\stn}, \chi_{\Ind^\sfG_\sfH(\varphi)}\rangle \ , $$
that is, there are either $0$ or a multiple of $u(\lambda)$ many arrows from $\lambda$ to $\Ind^\sfG_\sfH(\varphi)$ in $\Gamma(\sfG)$.
\end{lemma}

\begin{proof}
By Prop.~\ref{Prop:Stembridge},  $\varphi$ is of the from $([\phi],\delta)$ for an irreducible representation $\phi$ of $\sfG$ and $\delta \in (\sfG/\sfH)_\phi$. By \ref{Prop:ind} of the same proposition, $\Ind^\sfG_\sfH(\varphi)=\oplus \mu$, where the sum runs over all distinct irreducible $\sfG$-representations $\mu$ with $\mu \simeq \phi$. Thus we can write
$$\langle \chi_{\lambda} \cdot \chi_{\stn}, \chi_{\Ind^\sfG_\sfH(\varphi)}\rangle=\sum_{\mu \simeq \phi} \langle \chi_{\lambda} \cdot \chi_{\stn}, \chi_{\mu}\rangle \ .$$
By Lemma \ref{Lem:ResHnonzero}, if $\langle \chi_{\lambda} \cdot \chi_{\stn}, \chi_\mu \rangle \neq 0$, then $\Res^\sfG_\sfH(\mu)$ is irreducible in $\sfH$, which implies $u(\mu)=1$. If now moving a single cell of $\lambda$ can result in a summand of $\Ind^\sfG_\sfH(\varphi)$, then there will be $u(\lambda)$ many constituents of $\lambda \otimes \stn$ that correspond to a distinct $\mu \simeq \phi$. Since all these $\mu$ appear and $u(\mu)=1$, we are done. 
\end{proof}

\begin{Thm} \label{Thm:McKayGrpn}
Let $\sfG=G(r,1,n)$ and $\sfH=G(r,p,n)$ with $p >1$ as above. Further let $\Gamma(\sfH)$ be the McKay quiver of $\sfH$. Let $([\alpha],\delta)$ and $([\beta],\delta')$ be two irreducible representations of $\sfH$. Then there is an arrow with source $([\alpha],\delta)$ and target $([\beta],\delta')$ in $\Gamma(\sfH)$ whenever $[\beta]$ can be obtained from $([\alpha],\delta)$ by moving a single cell contained in the fundamental domain of $[\alpha]$ cyclically to the right. 
\end{Thm}

\begin{proof}
Note that $\Res^\sfG_\sfH(V_{\stn})=[\stn]$ remains irreducible and also remains the standard representation in $\sfH$. Let $\alpha$ be an irreducible representation of $\sfG$ and let $([\beta],\delta')$ be an irreducible representation of $\sfH$. This means that $\beta \in \Irrep(\sfG)$ and $\delta' \in (\sfG/\sfH)_\beta$. 
The arrows in $\Gamma(\sfH)$ are determined by Frobenius reciprocity. We have
\begin{equation} \label{Eqn:FrobGrpn}
\langle \chi_\alpha \cdot \chi_{\stn}, \chi_{\Ind^\sfG_\sfH([\beta],\delta')}\rangle=\langle \chi_{\Res^\sfG_\sfH(\alpha \otimes \stn)},\chi_{[\beta],\delta'}\rangle=\langle \chi_{\Res^\sfG_\sfH(\alpha)} \cdot \chi_{\Res^\sfG_\sfH(\stn)},\chi_{[\beta],\delta'}\rangle \ .
\end{equation}
We distinguish two cases, depending on whether $u(\alpha)=1$ or $u(\alpha) >1$, that is, whether the restriction of $\alpha$ to $\sfH$ is irreducible or it splits.

{\bf Case 1:} Assume that $u(\alpha)=1$. 
Then $\Res^\sfG_\sfH(\alpha)=[\alpha]$ is irreducible. The induced representation $\Ind^\sfG_\sfH([\beta],\delta')$ consists of the direct sum of all distinct $\beta^{\shift{i \cdot d}}, i=0, \ldots p-1$. From \eqref{Eqn:FrobGrpn} we know that there will be an arrow from $[\alpha]$ to $([\beta],\delta')$ in $\Gamma(\sfH)$ whenever there is an arrow from $\alpha$ to one of the components of $\Ind^\sfG_\sfH([\beta],\delta')$ in $\Gamma(\sfG)$. By Theorem \ref{thm:ML} this is the case if and only if one of the $\beta^{\shift{i \cdot d}}$ can be obtained from the $r$-tuple of $\alpha$ by moving a cell from position $i$ in $\alpha$ to position $i+1 \mod r$. Since $u(\alpha)=1$, after restriction $\alpha$ remains irreducible and the description of arrows is the same: there is an arrow from $[\alpha]$ to $([\beta],\delta')$ if and only if the $r$-tuple associated to $[\beta]$ can be obtained by moving a cell from $[\alpha]$ from position $i$ to position $i+1 \mod r$.

{\bf Case 2:} Assume that $u(\alpha)>1$. Then $\Res^\sfG_\sfH(\alpha)$ splits in $\sfH$ as $\oplus_{\delta \in (\sfG/\sfH)_\alpha}([\alpha],\delta)$. Then equation \eqref{Eqn:FrobGrpn} reads as 
$$\langle \chi_\alpha \cdot \chi_{\stn}, \chi_{\Ind^\sfG_\sfH([\beta],\delta')}\rangle=\sum_{i=1}^{u(\alpha)}\langle \chi_{[\alpha], (\delta_1^d)^i} \cdot \chi_{[\stn]}, \chi_{[\beta],\delta'}\rangle \ ,$$
where $\delta_1^d$ is the generator of $(\sfG/\sfH)_\alpha$.
By Lemma \ref{Lem:symmetryDiv} $u(\alpha)$ divides the left hand side. So there are $u(\alpha)$ arrows indicated by the first pairing to be shared equally among each $([\alpha],(\delta_1^d)^i)$ in the last pairing. Note that each of these irreducible representations has the same associated $r$-tuple $[\alpha]$ but a different fundamental domain. We can think of this action as moving a single available cell in the fundamental domain of $([\alpha],(\delta_1^d)^i)$ cyclically to the right.
\end{proof}

\begin{ex}
Let $[\alpha]=[(\ydiagram{1},\ydiagram{2,1},\ydiagram{1,1},-)]$ be an irreducible representation of $G(4,4,6)$. Then
\begin{align*}[\alpha] \otimes [\stn]  =&[(-,\ydiagram{2,2},\ydiagram{1,1},-)] \oplus [(-,\ydiagram{3,1},\ydiagram{1,1},-)]\oplus [(-,\ydiagram{2,1,1},\ydiagram{1,1},-)]\oplus [(\ydiagram{1},\ydiagram{2},\ydiagram{2,1},-)] \oplus \\
& [(\ydiagram{1},\ydiagram{2},\ydiagram{1,1,1},-)] \oplus [(\ydiagram{1},\ydiagram{1,1},\ydiagram{2,1},-)] \oplus [(\ydiagram{1},\ydiagram{1,1},\ydiagram{1,1,1},-)] \oplus [(\ydiagram{1} ,\ydiagram{2,1},\ydiagram{1},\ydiagram{1})] \ .
\end{align*}
We see that there are $8$ arrows in $\Gamma(G(4,4,6))$ with source $[\alpha]$. The targets of these arrows are the ones appearing in the direct sum calculated above.
\end{ex}

\begin{ex} \label{Ex:loop}
Let $[\alpha]=[(\ydiagram{1},\ydiagram{1,1},\ydiagram{1})]$ be an irreducible representation of $G(3,3,4)$. Again we compute the arrows in $\Gamma(G(3,3,4))$ with $[\alpha]$ as source via
\begin{align*}
[\alpha]\otimes [\stn]= & [(\ydiagram{1},\ydiagram{1},\ydiagram{1,1})]\oplus [(\ydiagram{1},\ydiagram{1},\ydiagram{2})]\oplus [(\ydiagram{1,1},\ydiagram{1,1},-)] \oplus
 [(\ydiagram{2},\ydiagram{1,1},-)] \oplus [(-,\ydiagram{1,1,1},\ydiagram{1})]\oplus [(-,\ydiagram{2,1},\ydiagram{1})] \ . 
\end{align*}
Here we see that the first summand on the right hand side is equal to $[\alpha]$. This means that there is a loop in the McKay quiver.
\end{ex}

\begin{ex} \label{Ex:doublearrow}
Let $[\alpha]=[(\ydiagram{2},\ydiagram{1},-)]$ be an irreducible representation of $G(3,3,3)$. Tensoring with $[\stn]$ yields again the arrows in the McKay quiver with source $[\alpha]$:
$$[\alpha] \otimes [\stn]= [(\ydiagram{1},\ydiagram{2},-)]\oplus [(\ydiagram{1},\ydiagram{1,1},-)] \oplus [(\ydiagram{2},-,\ydiagram{1})] \ .$$
Since $[(\ydiagram{1},\ydiagram{2},-)]=[(\ydiagram{2},-,\ydiagram{1})]$, there is a double arrow from $[\alpha]$ to this representation in $\Gamma(G(3,3,3))$.
\end{ex}

Examples \ref{Ex:loop} and \ref{Ex:doublearrow} show the following

\begin{cor} \label{Cor:loopsmultiple}
Let $\sfH=G(r,p,n)$ with $p>1$. Then the McKay quiver $\Gamma(\sfH)$ may have loops and multiple arrows from a vertex to another. 
\end{cor}

Let us now give examples where irreducible representations with nontrivial stabilizer $(\sfG/\sfH)_\alpha$ appear.

\begin{ex}
Let $([\alpha],0)=[\underline{\ydiagram{1}}, \ydiagram{1}, \ydiagram{1}]$ be an irreducible representation of $G(3,3,3)$. We compute again the arrows with source $([\alpha],0)$:
$$[\underline{\ydiagram{1}}, \ydiagram{1}, \ydiagram{1}] \otimes [\stn]=[-, \ydiagram{1,1}, \ydiagram{1}] \oplus [-, \ydiagram{2}, \ydiagram{1}] \ .$$
We see that there is only one cell in the fundamental domain of $[\alpha]$ and there are two ways to move this cell cyclically to the right.
\end{ex}

\begin{ex} In this examples we will see that also a target of an arrow in the McKay quiver may have nontrivial stabilizer. If this occurs than all conjugates of this representation will also appear as targets of arrows with the same source. Take $[\alpha]=[(-,\ydiagram{2},\ydiagram{1},\ydiagram{1})]$, an irreducible representation of $G(4,4,4)$. Then
\begin{align*}[\alpha]\otimes[\stn]=&[(-,\ydiagram{1},\ydiagram{2},\ydiagram{1})] \oplus [(-,\ydiagram{1},\ydiagram{1,1},\ydiagram{1})] \oplus \\
&  [(\underline{-,\ydiagram{2}}, -, \ydiagram{2})] \oplus [(-,\ydiagram{2},\underline{-,\ydiagram{2}})] \oplus [(-,\ydiagram{2},-,\ydiagram{1,1})]\oplus [(\ydiagram{1},\ydiagram{2},\ydiagram{1},-)] \ . 
\end{align*}
Note here that by Lemma \ref{Lem:ResHnonzero} we will never have arrows between two irreducible representations that both have nontrivial stabilizer.
\end{ex}

\begin{ex}[$p \neq r$] Let $([\alpha],0)$ be the irreducible representation $[(\underline{(\ydiagram{2},\ydiagram{1})},(\ydiagram{2},\ydiagram{1}))]$ of $G(4,2,6)$. Here we are only allowed to move the cells in the first tuple, since $d=2$. Thus we get $4$ arrows with source $([\alpha],0)$:
$$([\alpha],0) \otimes [\stn]=[((\ydiagram{1},\ydiagram{2}),(\ydiagram{2},\ydiagram{1}))]\oplus [((\ydiagram{1},\ydiagram{1,1}),(\ydiagram{2},\ydiagram{1}))]\oplus [((\ydiagram{2},-),(\ydiagram{3},\ydiagram{1}))]\oplus [((\ydiagram{2},-),(\ydiagram{2,1},\ydiagram{1}))] \ . $$
\end{ex}

\begin{example} \label{Ex:Coxeter2}
  Here we use Theorem \ref{Thm:McKayGrpn} to calculate the McKay quivers of the Coxeter groups $G(p,p,2)$, also known as the dihedral groups $I_2(p)$, that is, the symmetry group of a regular $p$-gon. Note that the discriminant of $G(p,p,2)$ is an $A_{p-1}$-curve singularity with coordinate ring $R/(\Delta)$, where $R=K[[x,y]]$ and $\Delta=y^2+x^p$. Thus one already knows the McKay quiver: it is just the Auslander--Reiten quiver of $\CM(R/\Delta)$ (plus an extra vertex corresponding to the determinantal representation). These AR-quivers have been described e.g. in \cite{Yoshino,GonzalesSprinbergVerdier}.
  
Let now $p=2k+1$ for $k \geq 1$. There are $k+2$ irreducible representations of $G(2k+1,2k+1,2)$ and none of them has any symmetry. They are given by the following vectors $\lambda_1=[(\ydiagram{1,1} \ ,\underbrace{-}_{2k})]$, $\lambda_2=[(\ydiagram{2} \ ,\underbrace{-}_{2k})]$ and $\lambda_i=[(\ydiagram{1} \ ,\underbrace{-}_{i-3}, \ydiagram{1} \ , \underbrace{-}_{2k-i+2})]$ for $i=3, \ldots, k+2$. Note that $\lambda_{\stn}=\lambda_3$. One easily sees that the McKay quiver looks as follows: 
\[
\begin{tikzpicture}
\node at (4,0) {\begin{tikzpicture} 
\node (C1) at (0,1)  {$\lambda_1$};
\node (C1a) at (0,-1)  {$\lambda_2$};
\node (C2) at (1.75,0)  {$\lambda_{\stn}$};
\node (C3) at (3.5,0)  {$\lambda_4$};
\node (C4) at (5.25,0)  {$\cdots$};
\node (C5) at (7,0)  {$\lambda_{k+1}$};
\node (C6) at (8.75,0)  {$\lambda_{k+2}$};

\draw [->,bend left=20,looseness=1,pos=0.5] (C1) to node[]  {} (C2);
\draw [->,bend left=20,looseness=1,pos=0.5] (C2) to node[] {} (C1);

\draw [->,bend left=20,looseness=1,pos=0.5] (C1a) to node[]  {} (C2);
\draw [->,bend left=20,looseness=1,pos=0.5] (C2) to node[] {} (C1a);

\draw [->,bend left=25,looseness=1,pos=0.5] (C2) to node[]  {} (C3);
\draw [->,bend left=25,looseness=1,pos=0.5] (C3) to node[] {} (C2);

\draw [->,bend left=25,looseness=1,pos=0.5] (C3) to node[]  {} (C4);
\draw [->,bend left=25,looseness=1,pos=0.5] (C4) to node[] {} (C3);

\draw [->,bend left=25,looseness=1,pos=0.5] (C4) to node[]  {} (C5);
\draw [->,bend left=25,looseness=1,pos=0.5] (C5) to node[] {} (C4);

\draw [->,bend left=25,looseness=1,pos=0.5] (C5) to node[]  {} (C6);
\draw [->,bend left=25,looseness=1,pos=0.5] (C6) to node[] {} (C5);
\draw[->]  (C6) edge [in=25,out=-25,loop ,looseness=10,pos=0.5] node[left] {} (C6);

\end{tikzpicture}}; 
\end{tikzpicture}
\] 

Note that there is a loop at $\lambda_{k+2}$.

 Similarly, for $G(2k,2k,2)$ there are in total $k+3$ irreducible representations, and only two of them have a symmetry. The irreducible representations are 
 $$\lambda_1=[(\ydiagram{1,1} \ ,\underbrace{-}_{2k-1})], \lambda_2=[(\ydiagram{2} \ ,\underbrace{-}_{2k-1})] \text{ and } \lambda_i=[((\ydiagram{1} \ ,\underbrace{-}_{i-3}, \ydiagram{1} \ , \underbrace{-}_{2k-i+1})] \ , $$
  for $i=3, \ldots, k+1$ and moreover 
  $$\lambda_{k+2}=[(\underline{\ydiagram{1} \ ,\underbrace{-}_{k-1}}, \ydiagram{1} \ , \underbrace{-}_{k-1})] \text{ and }\lambda_{k+3}=[(\ydiagram{1} \ ,\underbrace{-}_{k-1}, \underline{\ydiagram{1} \ , \underbrace{-}_{k-1}})] \ . $$
   Again it is easy to see that the McKay quiver looks as follows: 
   
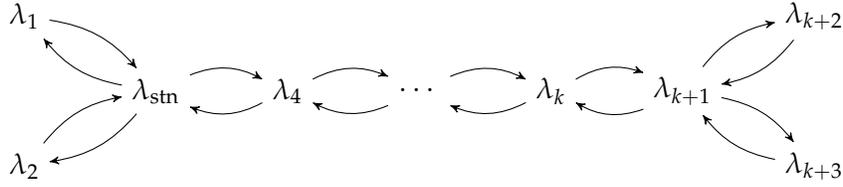
\[
\begin{tikzpicture}
\node at (4,0) {\begin{tikzpicture} 
\node (C1) at (0,1)  {$\lambda_1$};
\node (C1a) at (0,-1)  {$\lambda_2$};
\node (C2) at (1.75,0)  {$\lambda_{\stn}$};
\node (C3) at (3.5,0)  {$\lambda_4$};
\node (C4) at (5.25,0)  {$\cdots$};
\node (C5) at (7,0)  {$\lambda_{k}$};
\node (C6) at (8.75,0)  {$\lambda_{k+1}$};
\node (C7) at (10.5,1)  {$\lambda_{k+2}$};
\node (C7a) at (10.5,-1)  {$\lambda_{k+3}$};

\draw [->,bend left=20,looseness=1,pos=0.5] (C1) to node[]  {} (C2);
\draw [->,bend left=20,looseness=1,pos=0.5] (C2) to node[] {} (C1);

\draw [->,bend left=20,looseness=1,pos=0.5] (C1a) to node[]  {} (C2);
\draw [->,bend left=20,looseness=1,pos=0.5] (C2) to node[] {} (C1a);

\draw [->,bend left=25,looseness=1,pos=0.5] (C2) to node[]  {} (C3);
\draw [->,bend left=25,looseness=1,pos=0.5] (C3) to node[] {} (C2);

\draw [->,bend left=25,looseness=1,pos=0.5] (C3) to node[]  {} (C4);
\draw [->,bend left=25,looseness=1,pos=0.5] (C4) to node[] {} (C3);

\draw [->,bend left=25,looseness=1,pos=0.5] (C4) to node[]  {} (C5);
\draw [->,bend left=25,looseness=1,pos=0.5] (C5) to node[] {} (C4);

\draw [->,bend left=25,looseness=1,pos=0.5] (C5) to node[]  {} (C6);
\draw [->,bend left=25,looseness=1,pos=0.5] (C6) to node[] {} (C5);

\draw [->,bend left=20,looseness=1,pos=0.5] (C6) to node[]  {} (C7);
\draw [->,bend left=20,looseness=1,pos=0.5] (C7) to node[] {} (C6);

\draw [->,bend left=20,looseness=1,pos=0.5] (C6) to node[]  {} (C7a);
\draw [->,bend left=20,looseness=1,pos=0.5] (C7a) to node[] {} (C6);

\end{tikzpicture}}; 
\end{tikzpicture}
\] 

\end{example}

These examples already suggest that the McKay graphs of $G(r,p,n)$ for $p>1$ have less symmetry than the ones for
$G(r,1,n)$. We make this more precise.

First note that the abelianization $G^{\mathrm{ab}}=G/[G,G]$ of
  $G(r,1,n)$ is $\mu_2 \times \mu_{r}.$ For $p|r$ we have that
  $G(r,1,n)^{\mathrm{ab}}$
  surjects onto $\mu_p$ and so $G(r,p,n)^{\mathrm{ab}} = \mu_2\times \mu_{r/p}$.

\begin{proposition}\label{symGroup}
  Let $\sfG \subseteq \GL(V)$ be a finite group with McKay quiver $\Gamma(\sfG)$.
  Then the group of one dimensional representations $\Hom(\sfG/[\sfG,\sfG],K^*)$ acts faithfully on $\Gamma(\sfG)$.
\end{proposition}
\begin{proof}
  Note first that $\Hom(\sfG/[\sfG,\sfG],K^*)$ is naturally the one dimensional representations of $\sfG$ with tensor product as group operation.
  Let $L$ be a one dimensional representation of $\sfG$. It acts on $\Gamma(\sfG)$
  by sending the irreducible representation $W$ to
  the irreducible representation $L\otimes W$.
  Since 
  $$\Hom_{K\sfG}(U,W\otimes V ) \cong \Hom_{K\sfG}(L\otimes U, L \otimes W\otimes V) \ , $$ this action preserves the arrows of $\Gamma(\sfG)$.  Also, if $K$ is the trivial
  representation, then $L\otimes K \cong K$ if and only if $L\cong K$ so the action is faithful.
\end{proof}

It would be interesting to answer the following question.
\begin{Qu}
  Describe the automorphism group of the McKay quiver $\Gamma(G(r,p,n)).$
\end{Qu}

\section{Lusztig algebras} \label{Sec:Lusztig}

In this section we describe the Lusztig algebra $\widetilde{A}(\sfG,W)$ of a finite group $\sfG$ over an algebraically closed field $K$ with respect to a $\sfG$-$K$-algebra $A$ and a $\sfG$-representation $W$. This algebra was considered by \cite[Section 6]{LusztigAdvances} in the context of quiver varieties. Recall that the basic algebra of $\sfG * A$ is the path algebra of the McKay quiver of $\sfG$ modulo relations.  The Lusztig algebra presents this basic algebra as an invariant ring of a matrix algebra
- this will be done explicitly for several examples in Section \ref{Sub:Examples}.

The study of Lusztig algebras $\widetilde{A}(\sfG,W)$ is interesting in its own right, as it provides some new algebras obtained from finite groups $\sfG$: for $W$ a representation generator, one obtains through varying the (Koszul) algebra $A$ different relations on the path algebra $KQ$ of the McKay quiver $Q$ of $\sfG$. These Lusztig algebras $\widetilde{A}(\sfG,W)$ seem to encode some not yet explored structures of the representation theory of $\sfG$. 

\subsection{Construction}

We fix a finite group $\sfG$ over an algebraically closed field $K$, with $\Irrep(\sfG)=\{V_i\}_{i=1}^r$.
We call $T=\oplus_{i=1}^{r}V_{i}$ the \emph{basic} $\sfG$--representation. It is unique up to 
isomorphism and the $K\sfG$--module structure is the obvious one coming from the action of 
$\prod_{i=1}^{r}\End_{K}(V_{i})$ on this direct sum. The term ``basic'' is chosen as by Schur's
Lemma there are algebra isomorphisms 
$\End_{K\sfG}(T)\cong \prod_{i=1}^{r}\End_{K\sfG}(V_{i})\cong K^{r}$, with 
componentwise operations on the latter.

The basic representation can be realized as a direct summand of the regular representation
$K\sfG\cong \oplus_{i}(V_{i}\otimes V^*_i)$ of $\sfG$, however, there is generally no canonical 
such choice. On each summand $\End_{K}(V_{i})\cong V_{i}\otimes V_{i}^{*}$ of the regular representation,
$\sfG$ acts though the left factor and singling out a copy of $V_{i}$ is equivalent to 
choosing a nonzero element $\lambda_{i}\in V_{i}^{*}$, equivalently, a projection
on $V_{i}^{*}$ with one--dimensional image. Such a projection, 
viewed in $K\sfG$, is represented 
by an idempotent $e_{i}=e_{\lambda_{i}}\in K\sfG$ so that $V_{i}\cong K\sfG e_{i}$ as
$\sfG$--representation. For $i\neq j$, the chosen idempotents are orthogonal, whence
$e=\sum_{i}e_{i}\in K\sfG$ is still idempotent and $T\cong K\sfG(\sum_{i}e_{i})$.

As $T$ contains a copy of every irreducible representation (up to isomorphism) as a direct summand,
it represents a progenerator for $K\sfG$ and affords thus the Morita equivalence between
$K\sfG$ and $\End_{K\sfG}(T)\cong K^{r}$ (cf.~e.g.~\cite[\S 5]{McConnellRobson}). Note that this in particular means that $K{\sfG}eK{\sfG}=K\sfG$.

Consider a $\sfG$--$K$--algebra $A$, that is a $K$--algebra $A$ together with 
a group homomorphism $\alpha\colon \sfG\to \Aut_{K-alg}(A)$. 
We denote by $A*\sfG$ the corresponding skew group algebra of $A$ on $\sfG$.
The multiplication for any $ag, a'g' \in A*\sfG$ is given as
$$ ag \cdot a'g' = a\alpha(g)(a') (gg') \ , $$ and then linearly extended to any elements in this ring. If $\alpha$ is understood, we shorten the notation to $g(a)$ instead of $\alpha(g)(a)$. 
Note that $A*\sfG \cong \sfG *A$ since $ag=g g^{-1}(a)$. 

Let $f:K\sfG\to \sfG*A$ denote the canonical $K$--algebra homomorphism from the group
algebra on $\sfG$ to the twisted group algebra that sends $g\in \sfG$ to $g\otimes 1=g$
in $\sfG*A$.

As for any ring homomorphism, $f$ defines an adjoint triple of functors between the 
respective module categories. (We prefer to deal with right modules.)
\begin{align*}
f_{*}&\colon\Mod \sfG*A\lto \Mod K\sfG&&\text{is the forgetful functor,}\\
f^{*}&\colon \Mod K\sfG\lto \Mod \sfG*A &&\text{is the left adjoint to $f_{*}$,}\\
f^{!}&\colon \Mod K\sfG\lto \Mod \sfG*A &&\text{is the right adjoint to $f_{*}$,}
\end{align*}
Because $K\sfG$ is semisimple, these functors are quite simple. We note the following.

\begin{lemma}
For every (right) $K\sfG$--module $W$ one has 
\begin{align*}
f^{*}(W) = W\otimes_{K\sfG}\sfG*A\cong W\otimes A
\end{align*}
with the right $\sfG*A$--module structure determined by
$(w\otimes a)(g\otimes a') = (w\cdot g)\otimes \alpha(g^{-1})(a)a'$.

The $\sfG*A$--modules $f^{*}W$ are projective and $f^{*}(\oplus_{i}V_{i})$ is a progenerator.

\end{lemma}
\begin{sit}
Attach to such a triple $(\sfG, A, W)$ the $K$--algebra $\End_{K}(W)\otimes A$, tensor product of two 
$K$--algebras with the usual induced multiplication coming from the two factors.
In other words, the multiplication is not twisted as, say, in the case of $A*\sfG$, the twisted
tensor product of $A$ and $K\sfG$.

The group $\sfG$ acts (from the left) on this tensor product through $K$--algebra automorphisms
in the obvious way:
\begin{align*}
g(\vp\otimes a) &= (\rho(g)\vp\rho(g)^{-1})\otimes \alpha(g)(a)\,,
\end{align*}
where $\rho(g)\vp\rho(g)^{-1}(w) = \rho(g)(\vp(\rho(g)^{-1}(w)))$ is the transformed 
$K$--linear endomorphism of $\vp$ in $\End_{K}(W)$ defined by the standard $\sfG$--action on 
that endomorphism ring.

\end{sit}

\begin{defi} \label{Def:LusztigAlg}
Given a triple $(\sfG, A, W)$, we associate to it the $K$--algebra
\[
\widetilde A(\sfG,W) =\left(\End_{K}(W)\otimes A\right)^{\sfG}\,
\]
where the $\sfG$--fixed points are with respect to the action of $\sfG$ just described.

If $T=\oplus_{i=1}^{r}V_{i}$ is the direct sum of a set of representatives of the 
isomorphism classes of finite dimensional irreducible $\sfG$--representations of $\sfG$, then we call
\begin{align*}
\widetilde A(\sfG) = \widetilde A(\sfG,T)
\end{align*}
the \emph{Lusztig algebra\/} of $\sfG$ on $A$. 
\end{defi}

This algebra was first described by Lusztig \cite[6.1]{LusztigAdvances}.

As a first indication as to the size of these algebras, note the following.
\begin{lemma}
For every triple $(\sfG,A,W)$, the $K$--algebra $\End_{K\sfG}(W)\otimes A^{\sfG}$ is a
subalgebra of $\widetilde A(\sfG,W)$. \qed
\end{lemma}
By Schur's Lemma, if $W\cong \oplus_{i=1}^{r}V_{i}^{m_{i}}$ is a decomposition of 
$W$ as direct sum of the different irreducible $\sfG$--representations with multiplicities, then 
$\End_{K\sfG}(W)
\cong 
\prod_{i}\Mat_{m_{i}\times m_{i}}(K)$ as
$K$--algebras.

We give different representations of these algebras.

\begin{sit}
First note that as $K$--vector spaces one has natural isomorphisms
\begin{equation}
\label{dis1}
\End_{K}(W)\otimes A\cong \Hom_{K}(W, W\otimes A)\cong W\otimes A\otimes W^{*}
\end{equation}
as $W$ is a finite dimensional $K$--vector space. 

If $\dim_{K}W=d$, then $W\otimes A\otimes W^{*}$ 
can be identified with the set of $d\times d$ matrices $\Mat_{d\times d}(A)$ 
with entries from $A$ in that such a matrix $(a_{ij})_{i,j=1,...d}$ represents 
$\sum_{i,j}w_{i}\otimes a_{ij}\otimes w_{j}^{*}$ for a fixed ordered basis 
$(w_{1},..., w_{d})$ of $W$ and its dual basis $(w_{1}^{*},..., w_{d}^{*})$ of $W^{*}$.
The multiplication is then just the standard multiplication of matrices,
\[
(a_{ij})_{i,j=1,...d}(a'_{jk})_{j,k=1,...d} =\left(\sum_{j=1}^{d}a_{ij}a'_{jk}\right)_{_{i,k=1,...d}}\,.
\] 
On these matrices, the action of $g\in \sfG$ is given by
\begin{equation} \label{eq:Gaction}
g(a_{ij})= \rho(g)(\alpha(g)(a_{ij}))\rho(g)^{-1}\,.
\end{equation}
Even more explicitly: 
\end{sit}

\begin{sit}
If $M=(a_{ij})$ is in $\Mat_{d\times d}(A)$, then for every $g\in \sfG$ taking traces $\tr$ yields that
$\tr(g(M)) = \alpha(g)(\tr(M))$. 
Accordingly, if the matrix is in the fixed ring  $\Mat_{d\times d}(A)^{\sfG}$, then the 
trace of $M$ must be an invariant element of $A$. Using that $d=\dim W$ is invertible in $K$,
the matrix $M'=M-\frac{1}{d}\tr(M)\id_{d}$ is in $\frak{sl}(d,A)^{\sfG}$, the 
invariant traceless $d\times d$ matrices over $A$. Conversely, $M$ is invariant if $M'$ is
and $\tr(M)\in A^{\sfG}$. 
\end{sit}

For determining $\Mat_{d\times d}(A)^{\sfG}$ it thus suffices to determine 
$A^{\sfG}$ and $\frak{sl}(d,A)^{\sfG}$.

\begin{sit}
The isomorphisms in (\ref{dis1}) above imply
\begin{equation} \label{Eq:G-invEndos}
\widetilde A(\sfG,W) = (\End_{K}(W)\otimes A)^{\sfG} \cong \Hom_{K\sfG}(W, W\otimes A)\,.
\end{equation}

In a basis free manner, if $\vp',\vp\colon W\to W\otimes A$ are $K$--linear or $K\sfG$--linear, 
then their product is given by the composition (cf.~\cite[6.1]{LusztigAdvances})
\begin{align*}
\vp'\vp\equiv W\xto{\ \vp\ } W\otimes A\xto{\ \vp'\otimes\id_{A}\ }W\otimes A\otimes A
\xto{\ \id_{W}\otimes \mu\ }W\otimes A\,,
\end{align*}
where $\mu:A\otimes A\to A$ is the multiplication on $A$.

Indeed, with respect to a chosen basis of $W$ as before, write
$\vp(w_{k})= \sum_{j}w_{j}\otimes a_{jk}$ and 
$\vp'(w_{j})= \sum_{i}w_{i}\otimes a'_{ij}$ with uniquely determined elements $a_{ij}, a'_{ij}\in A$. 
One then has
\begin{align*}
(\id_{W}\otimes \mu)(\vp'\otimes \id_{A})\vp(w_{k})&=
(\id_{W}\otimes \mu)(\vp'\otimes \id_{A})(\sum_{j}w_{j}\otimes a_{jk})\\
&=(\id_{W}\otimes \mu)\left(\sum_{j}\vp'(w_{j})\otimes a_{jk} \right)\\
&=(\id_{W}\otimes \mu)\left(\sum_{j}\left(\sum_{i}w_{i}\otimes a'_{ij}\right)\otimes a_{jk} \right)\\
&=\sum_{i}w_{i}\otimes \sum_{j}a'_{ij}a_{jk}\,.
\end{align*}
\end{sit}

\begin{sit}
By definition, $K\sfG$ acts from the left on $W\otimes A$, while $A$ acts naturally from the right.
These actions combine to endow $W\otimes A$ with a right $A*\sfG$--module structure and 
one has the following.
\end{sit}

\begin{proposition}
There is an isomorphism of $K$--algebras 
\[
\End_{A*\sfG}(W\otimes A)\xto{\ \cong\ } 
(\End_{K}(W)\otimes A)^{\sfG}\,.
\]
\end{proposition}

\begin{proof}
Just restrict an endomorphism $\vp:W\otimes A\to W\otimes A$ to $W\otimes 1$ to get
the corresponding element of $\Hom_{K\sfG}(W,W\otimes A)\cong (\End_{K}(W)\otimes A)^{\sfG}$.
\end{proof}
Note that for any $W$, the $A*\sfG$--module $W\otimes A$ is a right projective $A*\sfG$--module.
If $T$ is a finite dimensional $K$--linear $\sfG$--representation that contains a 
representative of every irreducible $\sfG$--representation as a direct summand, then
$T\otimes A$ is a projective generator for $A*\sfG$ and thus we conclude:
\begin{corollary} \label{Cor:LusztigMoritaequiv}
If $T$ is a finite dimensional $K$--linear $\sfG$--representation that contains a 
representative of every irreducible $\sfG$--representation as a direct summand, then the
associated Lusztig algebra is Morita equivalent to $A*\sfG$.
\end{corollary}

\begin{Bem} \label{Rmk:Computationofdeg1}
In order to calculate $\widetilde{A}(\sfG,W)$ explicitly as a matrix algebra, we use \eqref{eq:Gaction} and \eqref{Eq:G-invEndos}: Let $d=\dim(W)$ and $M=(a_{ij})_{i,j=1, \ldots, d}$ be an element of $\End_K(W)\otimes A$. Then $M \in (\End_K(W) \otimes A)^\sfG$ if and only if
$$g(a_{ij})=\rho(g)(\alpha(g)(a_{ij}))\rho(g)^{-1} = a_{ij} \ \text{for all } i,j=1, \ldots, d \ , $$
where $\rho$ is the character of $W$. We write shorthand $M=\rho(g)(g(M))\rho(g)^{-1}$.

In particular, if $W=T=\sum_{i=1}^rV_i$, then $d=\sum_{i=1}^r \dim(V_i)$ and we may write $\rho=\sum_{i=1}^r \rho_i$ as the sum of irreducible representations. Further, one can calculate the $\dim(V_i) \times \dim(V_j)$ blocks $M^{ij}$ of $M$ corresponding to $\Hom_{K\sfG}(V_j, V_i \otimes A)$ via 
$$ M^{ij}=\rho_i(g) g(M^{ij}) \rho_j(g)^{-1} \ ,$$
for all $g \in \sfG$. Moreover, it is clearly enough to check these equalities on generators of $\sfG$. See Section \ref{Sub:Examples} for examples.
\end{Bem}

  \begin{Bem} The computation of the entries in the Lusztig algebra of $\sfG$ is reminiscent of Gutkin--Opdam matrices $M$ for a pseudo-reflection group $\sfG$  (cf. \cite[Section 4.5.2, p.80]{Brouebook}, also cf.~\cite{Beck, Opdam} : there $(K[V]\otimes W^*)^\sfG$ is calculated, and the matrix $M$ has to satisfy the relation \cite[Theorem 4.38 (i)]{Brouebook} $g(M)=\rho_W(g)M$ for all $g \in \sfG$. That construction gives a generalization of Stanley's semi-invariants for linear characters to characters of higher rank. In the special case when
    $W$ has the form $V \simeq \End_K(V) \simeq V\otimes V^*$ for some $\sfG$ representation $V$ and $A=K[V],$ then  these matrices are the matrices from Remark \ref{Rmk:Computationofdeg1}.
\end{Bem}

\subsection{Quiver and relations for $\widetilde{A}(\sfG)$} \label{Sub:QuiverRelationsKoszul}

In the following we keep the notations from above: $\sfG\subseteq \GL(V)$ is a finite group acting on $V$ and we denote the basic $\sfG$-representation by $T$. In order to be able to calculate the Lusztig algebra $\widetilde{A}(\sfG)$, we will assume that the $\sfG-K$-algebra $A$ is Koszul. We recall some notions of Koszul algebras, following mostly \cite{BeilinsonGinzburgSoergel}. 
A positively graded $K$-algebra $A=\bigoplus_{i \geq 0} A_i$ is called \emph{Koszul} if (1) $A_0$ is finite dimensional semi-simple and (2) $A_0$ considered as a graded (left) 
$A$-module admits a graded projective resolution
$$ \cdots \xrightarrow{} P_2 \xrightarrow{} P_1 \xrightarrow{} P_0 \xrightarrow{} A_0 \xrightarrow{} 0 \ ,$$
such that $P_i$ is generated in degree $i$. It is known \cite[Cor.~2.3.3]{BeilinsonGinzburgSoergel} that any Koszul ring is quadratic, that is, there is an isomorphism of graded algebras
$$A \cong T_{A_0}(V)/I=A_0 \oplus V \oplus (V \otimes V) \oplus \cdots \ , $$
where $V$ is an $A_0$-bimodule placed in degree $1$ and $I$ is is a two-sided ideal generated in degree $2$: $I=\langle M \rangle$, where $M$ is an $A_0$-bimodule in $V \otimes V$.
For example, let $Q$ be a quiver with vertices $Q_0$ and arrows $Q_1$.  Then let $A_0 = \prod_{Q_0} K$.  Then
  $M = \oplus_{Q_1}K$ is an $A_0$-bimodule and $kQ = T_{A_0} M$ graded by path length.

\begin{example} (1) $T_K(V)\cong K\langle x_1, \ldots, x_n \rangle$, the tensor algebra of the vector space $V$ (with basis $x_1, \ldots, x_n$ over a field $K$), is Koszul.  $T_K(V)$ is the free algebra in $n$ variables over $K$.
  
(2) $S=\Sym_K(V)\cong K[x_1, \ldots, x_n]$ is Koszul: it can be described as $T_K(V)/I$, where again $x_1, \ldots, x_n$ denotes the basis of $V$ and $I$ is the two-sided ideal generated by the commutators $x_ix_j-x_jx_i$. The Koszul complex of $K$ yields the linear graded projective resolution.  \\
(3) The exterior  algebra $\bigwedge_K(V^*)$ is Koszul: it can be described as $T_K(V^*)/I$, where $x^*_1, \ldots, x^{*}_n$ denote the dual coordinates on $V^*$ and $I$ is the ideal generated by $(x^{*}_i)^2$ for $i=1, \ldots, n$ and $x^*_ix^*_j + x^*_jx^*_i$ for $i \neq j$. In fact, it is well-known that $\bigwedge_K(V^*)$ is the Koszul dual to $\Sym_K(V)$, see e.g., \cite{Froeberg}. In the following examples we have $V\cong V^*$ as $\sfG$-representations, so we will sloppily write $\bigwedge_K(V)$ and not change to dual coordinates.
\end{example}

These three examples will be the paradigms for the algebra $A$. Moreover, if $A=S=\Sym_KV$ or $A=S^{!}=\bigwedge_K(V)$, then the skew group algebra $A*\sfG$ is Koszul. Note that for $S$ this is well-known and for $\bigwedge_K(V)$ a short proof can be found in \cite[Section 6.2,Cor.~3]{HuerfanoKhovanov}.  Moreover: if $A=T_K(V)$, then $A*\sfG$ is  Morita equivalent to the path algebra $KQ$ of the McKay quiver $Q=\Gamma(\sfG)$, see \cite[Thm.~1.3]{GuoMartinez}. These observations combine to the following

\begin{Thm} \label{Thm:LusztigAlgKoszul}
Let $A=T_K(V)/I$ be a positively graded algebra with $I \subseteq \bigoplus_{i \geq 2}T_K(V)_{i}$ and $\sfG$ be a finite subgroup of $\GL(V)$. Then $A*\sfG$ is Morita equivalent to a path algebra $KQ/\langle I \rangle$, where $Q$ is the McKay quiver of $\sfG$ and $\langle I \rangle$ is the two-sided ideal in $KQ$ induced by the relations of $I$. In particular: the Lusztig algebra $\widetilde A(\sfG)\cong KQ/\langle I \rangle$ and it is Koszul if $A$ is Koszul. 
\end{Thm}

\begin{proof} For the relations, note that by \cite[Thm.~1.8]{GuoMartinez}, $A*\sfG=T_K(V)/I * \sfG \cong (T_K(V) * \sfG)/\langle I \rangle$, which is Morita equivalent to the basic algebra $KQ/\langle I \rangle$. Moreover, recall that by Cor.~\ref{Cor:LusztigMoritaequiv} the Lusztig algebra $\widetilde A(\sfG,T)$ is the basic version of $A*\sfG$, and thus $\widetilde A(\sfG,T)$ is isomorphic to  $KQ/\langle I \rangle$. 
For Koszulity,  note that the Morita equivalence is induced by the idempotent $e$  for $T$, which satisfies $(A*\sfG)e(A*\sfG)=A*\sfG$. This makes $e(A*\sfG)e=\widetilde A(\sfG,T)$ and then (see Lemma 2.2 of \cite{BocklandtSchedlerWemyss}) if $A*\sfG$ is Koszul, so is $\widetilde A(\sfG,T)$. 
\end{proof}

\subsection{Examples} \label{Sub:Examples}

\begin{example} \label{Ex:D4}
Consider $\sfG=D_4=\langle \alpha ,\beta \,|\, \alpha^4=\beta^2=(\alpha\beta)^2=1\rangle$ viewed as a subgroup of $\GL(2,K)$, with $K=\CC$ (see the table below for the expressions of the generators $\alpha$ and $\beta$). Note that as a subgroup of $\GL(2,K)$, this group is generated by reflections. 
Assume that $V\cong K^2$ has basis $x,y$. Here we calculate the Lusztig algebra $\widetilde A(\sfG)$ for $A=\Sym_K(V)=K[x,y]$ and for its Koszul dual $A^!=\bigwedge_K(V)$.

Therefore we first determine the McKay quiver $Q$ of $\sfG$ (degree $1$ part of the Lusztig algebra) and then the (necessarily) quadratic relations by looking at the degree $2$ part of $\widetilde A(\sfG)$. Then, using Theorem \ref{Thm:LusztigAlgKoszul}, $\widetilde A(\sfG) \cong KQ/\langle I \rangle$, where $I$ are the relations induced by the commutativity relation $xy-yx$ in $K\langle x , y \rangle$ (in case of $A=K[x,y]$) and the relations $xy+yx,x^2,y^2$ for the Koszul dual $A^!$.

 First note that $D_4$ has four one-dimensional irreducible representations $V_i$, $i=0, \ldots, 3$ and one two-dimensional one $V_4=V$ (the defining representation). Here $r=5$ and $d=6$ and the degree $1$ part of $\widetilde A(\sfG)$ is a matrix $M=(m_{ij})_{i,j=0,\ldots, 5}$ with entries in $\End_K(T) \otimes V \subseteq \End_K(T) \otimes T_K(V)$. It can be calculated using Remark \ref{Rmk:Computationofdeg1}: We order the representations as $V_0, \ldots, V_3$ and $V_4$ and thus $M^{ij}=m_{ij}$ encodes the arrow $j \xrightarrow{} i$ for $i,j=0, \ldots, 3$ whereas $M^{i4}=(m_{i4},m_{i5})$ encodes the arrows $4 \xrightarrow{} i$, $M^{4j}=(m_{4j},m_{5j})^T$ the arrows $j \xrightarrow{} 4$ for $i,j=0, \ldots, 3$, and $M^{44}=\begin{pmatrix}  m_{44} & m_{45} \\ m_{54} & m_{55} \end{pmatrix}$ the loops at $4$. Now for all $g \in \sfG$ we must determine for each entry $m_{kl}=a_{kl}x+b_{kl}y$, where $a_{kl}, b_{kl} \in K$,  in $M^{ij}$ such that 
$$M^{ij}=\rho_{V_i}(g)g(M^{ij})\rho_{V_j}^{-1}(g) \ . $$
Since it is enough to do this on the generators of the group, we only need the expressions $\rho_{V_i}$ of the two generators $\alpha,\beta$ (see table below). 
 \[
\begin{array}{|c|rr|}
\hline
g & \alpha & \beta \\
\hline
\rho_{V_0}(g) & 1 & 1 \\
\rho_{V_1}(g) & 1 & -1 \\
\rho_{V_2}(g) & -1 & 1 \\
\rho_{V_3}(g) & -1 & -1 \\
\rho_{V_4}(g) & \begin{pmatrix} i & 0 \\ 0 & -i \end{pmatrix} & \begin{pmatrix} 0 & 1 \\ 1 & 0 \end{pmatrix} \\
\hline
\end{array}
\]

For $i,j=0, \ldots 3$ we have $M^{ij}=m_{ij}$ and one gets from evaluating for $g=\beta$
$$\rho_{V_j}(\beta)\beta(m_{ij})\rho_{V_i}^{-1}(\beta)=\rho_{V_i}(\beta)\rho_{V_j}(\beta)\beta(m_{ij})=(\pm 1)(b_{ij}x-a_{ij}y) \ , $$
since the $\rho_{V_k}(\beta)$ are just $\pm 1$-scalars. This yields the equations
$$(\pm 1)b_{ij}=a_{ij} \text{ and } (\mp 1)a_{ij}=b_{ij} \ ,$$
which implies that all $a_{ij}=b_{ij}=0$. For the remaining entries in $M^{i4}$ (the degree $1$ part of $\Hom_{K\sfG}(V_4,V_i)$), $i=0,\ldots, 3$, we have to calculate the expression $\rho_{V_i}(g)g(m_{i4},m_{i5})\rho_{V_4}^{-1}(g)$ for $g=\alpha, \beta$ , which amounts to:
 \[
\begin{array}{|c|ccc|}
\hline
g & \rho_{V_i}(g)g(m_{i4},m_{i5})\rho_{V_4}^{-1}(g) & & \\
\hline
\alpha & \rho_{V_i}(\alpha)(ia_{i4}x-ib_{i4}y,ia_{i5}x-ib_{i5}y)\begin{pmatrix} -i & 0 \\ 0 & i \end{pmatrix}& = &\rho_{V_i}(\alpha)(a_{i4}x -b_{i4}y, -a_{i5}x+b_{i5}y) \\
\beta & \rho_{V_i}(\beta)(b_{i4}x+a_{i4}y,b_{i5}x+a_{i5}y) \begin{pmatrix} 0 & 1 \\ 1 & 0 \end{pmatrix} & = & \rho_{V_i}(\beta)(b_{i5}x+a_{i5}y,b_{i4}x+a_{i4}y) \\
\hline
\end{array}
\] 
Equating with $m_{i4},m_{i5}$ yields the expressions in columns $4$ and $5$ of the matrix $M$ in \eqref{Eq:D4}. Similarly, one obtains the four expressions for $\Hom_{K\sfG}(V_i,V_4)$ and $\Hom_{K\sfG}(V_4,V_4)$ can be seen to be the $2 \times 2$-zero matrix. Hence the degree $1$ part of $\widetilde A(\sfG)$ is matrices of the form
 \begin{equation} \label{Eq:D4} M= \begin{pmatrix} 0 & 0 & 0& 0 & ax & ay \\ 0 & 0 & 0 & 0 &bx & -by \\0 & 0 & 0 & 0 &cy & cx \\ 0 & 0 & 0 & 0 &dy & -dx \\
ey & fy &gx &hx & 0 & 0 \\ ex & -fx &gy &-hy & 0 & 0 \end{pmatrix} \ , 
\end{equation}

for arbitrary scalars $a,b,c,d,e,f,g,h$ in $K$. 
 The vertices of $Q$ correspond to the irreducible representations $V_i$ of $G$ and the
  the idempotent $e_i$ at the vertex $V_i$ corresponds to the block diagonal matrix with identity of size $\dim V_i$ in the appropriate location and zeroes elsewhere.  For example, $e_0$ is the single entry matrix $E_{00}$ and $e_4=E_{44}+E_{55}$.
 Now the arrows $i \xrightarrow{} j$ in the McKay quiver are given as $e_j M e_i$, $i,j=0, \ldots, 4$.  Note that setting the vectors
  $(a,b,c,d,e,f,g,h)$ to be the standard basis vectors corresponds to the 8 arrows of $Q$.  Thus we obtain from $M$ the following well-known McKay quiver $Q$ for $\sfG$. Note that the labels of the arrows $(A,\ldots,H)$ correspond to the scalars $(a, \ldots, h)$:

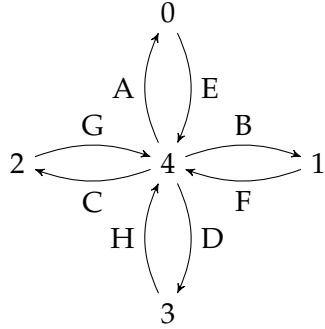
\[
\begin{tikzpicture}
\node at (4,0) {\begin{tikzpicture} 
\node (C1) at (0,0)  {$2$};
\node (C2) at (2,0)  {$4$};
\node (C3) at (4,0)  {$1$};
\node (C4) at (2,2)  {$0$};
\node (C5) at (2,-2)  {$3$};

\draw [->,bend left=20,looseness=1,pos=0.5] (C1) to node[above]  {G} (C2);
\draw [->,bend left=20,looseness=1,pos=0.5] (C2) to node[below] {C} (C1);

\draw [->,bend left=20,looseness=1,pos=0.5] (C3) to node[below]  {F} (C2);
\draw [->,bend left=20,looseness=1,pos=0.5] (C2) to node[above] {B} (C3);

\draw [->,bend left=25,looseness=1,pos=0.5] (C2) to node[left]  {A} (C4);
\draw [->,bend left=25,looseness=1,pos=0.5] (C4) to node[right] {E} (C2);

\draw [->,bend left=25,looseness=1,pos=0.5] (C2) to node[right]  {D} (C5);
\draw [->,bend left=25,looseness=1,pos=0.5] (C5) to node[left] {H} (C2);

\end{tikzpicture}}; 
\end{tikzpicture}
\]

For example $g=1$ and the the other scalars are zero gives the matrix corresponding to the arrow $G,$ 
$\begin{pmatrix} 0 & 0 & 0& 0 & 0 & 0 \\ 0 & 0 & 0 & 0 &0 & 0 \\0 & 0 & 0 & 0 &0 & 0 \\ 0 & 0 & 0 & 0 & 0 & 0 \\
0 & 0 &x & 0& 0 & 0 \\ 0 & 0 &y &0 & 0 & 0 \end{pmatrix}$.

Now we compute the relations: therefore first calculate $M \cdot M'$, where
$$M'= \begin{pmatrix} 0 & 0 & 0& 0 & a'x & a'y \\ 0 & 0 & 0 & 0 &b'x & -b'y \\0 & 0 & 0 & 0 &c'y & c'x \\ 0 & 0 & 0 & 0 &d'y & -d'x \\
e'y & f'y &g'x &h'x & 0 & 0 \\ e'x & -f'x &g'y &-h'y & 0 & 0 \end{pmatrix} $$ for scalars $a', b', c', d', e', f', g' \in K$ as this encodes all paths of length $2$ in $\widetilde A(\sfG)$ (assuming that $A$ is of the form $T_K(V)/I$), that is, a path of length $2$ starting at vertex $i$ and ending at vertex $j$ is $$\sum_{k: \exists i \xrightarrow{} k \xrightarrow{} j}\alpha_{k}e_jMe_kM'e_i $$
 for some $\alpha_k \in K$. In the example at hand, this matrix looks as follows:
\[ 
\begin{tiny} \begin{pmatrix} ae'(xy+yx) & af'(xy-yx) & ag' (x^2+y^2) & ah'(x^2 -y^2) & 0 & 0 \\ be'(xy-yx) & bf'(xy+yx) & bg'(x^2-y^2) & bh'(x^2+y^2) &0  & 0 \\ ce'(y^2+x^2) & cf'(y^2-x^2) & cg'(yx+xy) & ch'(yx-xy) &0 & 0 \\ de'(y^2-x^2) & df'(y^2+x^2) & dg'(yx-xy) & dh'(yx+xy) &0 & 0 \\
0 & 0 &0 &0 & (ea'+fb')yx + (gc'+ hd')xy & (ea'-fb')y^2 +(gc'-hd')x^2 \\ 0 & 0 &0 &0 & (ea'-fb')x^2+(gc'-hd')y^2 & (ea'+fb')xy + (gc'+ hd')yx \end{pmatrix} \end{tiny} \ . \]
For the case $A=\Sym_K(V)$, the commutativity relations immediately imply four relations from the upper left $4 \times 4$ submatrix, that is, 
$$AF=BE=CH=DG=0 \ .$$
 For the relation at $V_4$, one sees that for off-diagonal entries the conditions  $\alpha_{EA}EA - \alpha_{FB}FB=0$ and $\alpha_{GC}GC-\alpha_{HD}HD=0$ for scalars $\alpha_{EA}, \alpha_{FB}, \alpha_{GC}, \alpha_{HD}$ hold, and for the diagonal entries the commutativity relation yields $\alpha_{EA}EA+\alpha_{FB}FB+\alpha_{GC}GC+\alpha_{HD}HD=0$. Solving for the $\alpha_{XY}$ implies that we get one other relation at $V=V_4$, namely
$$EA+FB=GC+HD\ .$$
Thus the Lusztig algebra of $\sfG$ is
$$\widetilde{\Sym_K(V)}(\sfG) \cong KQ/\langle AF,BE,CH,DG,EA+FB-GC-HD \rangle . $$

For $A=\bigwedge_K(V)$, we immediately get relations 
$$AE=AG=AH=BF=BG=BH=CE=CF=CG=DE=DF=DH=0$$
from the upper left $4 \times 4$ submatrix. From $V_4$ we get the condition 

$\alpha_{EA}EA+\alpha_{FB}FB+\alpha_{GC}GC+(\alpha_{EA}+\alpha_{FB}-\alpha_{GC})HD=0$, which yields the relations 
$$EA+HD=0, FB+HD=0, GC-HD=0 \ .$$
Thus 
$\widetilde{\bigwedge(V)}(\sfG)\cong$ $$KQ/\langle AE,AG,AH, BF, BG, BH, CE, CF, CG, DE, DF, DH,EA+HD,FB+HD,GC-HD \rangle \  .$$
\end{example}

\begin{Bem}
Note that the quaternion group $\sfG = Q_8 = \langle \alpha, \beta \,|\, \alpha^4=1, \alpha^2=\beta^2, \beta \alpha=\alpha^{-1}\beta \rangle$ considered as a subgroup of $\GL(2,K)$ is in $\SL(2,K)$. Since $D_4$ and $Q_8$ have the same character table, the McKay quiver is then also $\widetilde D_4$ and similar calculations as above show that in this case, the Lusztig algebra $\widetilde{\Sym_K(V)}(\sfG)$ is isomorphic to the preprojective algebra of $\widetilde D_4$. Note that for subgroups $\sfG \subseteq \SL(2,K)$, one will always obtain that the Lusztig algebra  $\widetilde{\Sym_K(V)}(\sfG)$ is isomorphic to the preprojective algebra of the corresponding extended Dynkin diagram, see \cite{ReitenVdB}. 
 \end{Bem}

\begin{example} \label{Ex:abelian} 
  (The abelian case) Let $\sfG$ be a finite abelian subgroup of $\GL(V)$, with $\dim_{K}(V)=n$. First we calculate the McKay quiver of $\sfG$ and then the relations for the Lusztig algebra $\widetilde{S}(\sfG)$, where $S=\Sym_K(V) \cong K[x_1, \ldots, x_n]$. Note that $\widetilde{S}(\sfG)$ is isomorphic to the skew group ring $S*\sfG$ in this case. This example was first computed in \cite{Crawetc} and the relations with superpotentials in \cite{BocklandtSchedlerWemyss}. We follow mostly the notation in \cite{BocklandtSchedlerWemyss}, but note that the arrows in our McKay quiver go the other way round.
  
The setup is also the same as in loc.cit: Since $\sfG$ is abelian, we may choose a basis $x_1, \ldots, x_n$ of $V$ that diagonalizes the action of $\sfG$, and we get $n$ characters $\rho_1, \ldots, \rho_n$ defined via $\rho_i(g)$ the $i$-th diagonal element of $g \in \sfG$. Thus the defining representation of $\sfG$ is $\sum_{i=1}^n\rho_i$, but in total we have $|\sfG|$ many irreducible representations. We label them by $\rho_1, \ldots, \rho_{|\sfG|}$, so that the first $n$ are the ones defined above. First the description of the degree one part of $\widetilde{A}(\sfG)$: this is encoded in the $|\sfG| \times |\sfG|$-matrix $M$ where $m_{ij}$ is the arrow from $\rho_j$ to $\rho_i$ in the McKay quiver. Each $m_{ij}$ is of the form
$$m_{ij}=\sum_{l=1}^n a_{ij}^l x_l $$
with coefficients $a_{ij}^l \in S$. Being invariant under the $\sfG$-action means that for all $g \in \sfG$ we must have (cf.~Rmk.~\ref{Rmk:Computationofdeg1})
$$m_{ij}=\rho_j^{-1}(g) \rho_i(g) g(m_{ij}) \ .$$
Note that the $\rho_i(g)$ are just scalars here and clearly $g(x_l)=\rho_l(g)x_l$ for $l=1, \ldots, n$. That means that the $a_{ij}^l$ must satisfy the equations
$$a_{ij}^l\left( \rho_j^{-1}(g) \rho_i(g) \rho_l(g) -1 \right) = 0 $$
for all $g \in \sfG$. Thus we immediately see that for fixed $i,j$ there exists at most one $l \in \{1, \ldots, n\}$ such that $a_{ij}^l \neq 0$ (namely $l$ such that $\rho_j=\rho_i \otimes \rho_l$). Thus each entry $m_{ij}$ is a monomial $a_{ij}^lx_l$ or $0$. Fix $\rho_i$, then there are exactly $n$ arrows ending in $\rho_i$, namely starting at $\rho_j:=\rho_i \otimes \rho_l$ and with $m_{ij}=a_{ij}^l x_l$ for $l =1, \ldots, n$. This means that in the $i$-th row of $M$ we have $n$ non-zero entries, and each $x_l$ appears exactly once (If $m_{ij}=a_{ij}^l x_l\neq 0$ and $m_{ij'}= a_{ij'}^l x_l \neq 0$ then we would have $\rho_j=\rho_{j'}$). This gives the well-known description of the arrows 
$$ \rho_i \otimes \rho_l \xrightarrow{M_{ij}:=a_{ij}^lx_l} \rho_i\ $$ 
in the McKay quiver of $\sfG$.

Now we compute the relations in the Lusztig algebra $\widetilde{S}(\sfG)$, where $S=T_K(V)/\langle x_i x_j - x_j x_i, 1 \leq i < j \leq n \rangle$. We look again at $M \cdot M'$: $(M \cdot M')_{ij}$ are all paths of length $2$ from $\rho_j$ to $\rho_i$, that is $\sum_{k=1}^{|\sfG|}m_{ik}m'_{kj}$. Here $m_{ik}=a_{ik}^{l}x_l \neq 0$ if $\rho_k=\rho_i \otimes \rho_l$ and $m'_{kj}={a'}^{l'}_{kj}x_{l'} \neq 0$ if $\rho_j=\rho_k \otimes \rho_{l'}$. We can only get relations for $l \neq l'$. But if $l \neq l'$, then there exists a unique $k' \neq k$ such that $m_{ik'}m'_{k'j} \neq 0$, namely $\rho_{k'}:=\rho_i \otimes \rho_{l'}=\rho_j \otimes \rho_{l}^{-1}$. Thus, in order to get a relation in $\widetilde{S}(\sfG)$, the coefficient $\alpha_{k}$ of $M_{ik}M_{kj}$ must be the negative of the coefficient $\alpha_{k'}$ of $M_{ik'}M_{k'j}$. This yields the relations 
$$M_{ik}M_{kj} - M_{ik'}M_{k'j}=0 \ .$$
\end{example}

\begin{Bem}
For $\sfG=(\mu_2)^n$ acting on $K^n$ as a reflection group, this shows that the Lusztig algebra is isomorphic to the algebra of \cite{DaFI} and gives an alternative proof of \cite[Prop.~5.6]{DaFI}.
\end{Bem}

\begin{example} \label{Ex:S3}
Let $\sfG=S_{3}$, the symmetric group on three elements.  It is generated by the $2$--cycles
$s_{1}=(1,2), s_{2}=(2,3)$ and its irreducible representations can be identified as 
\[
\begin{array}{|c|cccccc|}
\hline
S_{3}&1& s_{1} & s_{2} & s_{12}=s_{1}s_{2} & s_{21}=s_{2}s_{1} &s_{3}=s_{1}s_{2}s_{1}=
s_{2}s_{1}s_{2} \\
\hline
\rho_{\triv} &1&1&1&1&1&1
\\
%\hline
\rho_{2,1}&\begin{pmatrix}
 1&0\\0&1
\end{pmatrix}&
\begin{pmatrix}
 0&1\\1&0
\end{pmatrix}&
\begin{pmatrix}
 0&\omega\\
 \omega^{2}&0
\end{pmatrix}&
\begin{pmatrix}
 \omega^{2}&0\\
 0&\omega
\end{pmatrix}&
\begin{pmatrix}
 \omega&0\\
 0&\omega^{2}
\end{pmatrix}&
\begin{pmatrix}
 0&\omega^{2}\\\omega&0
\end{pmatrix}\\
\rho_{\sgn} &1&-1&-1&1&1&-1\\
\hline
\end{array}
\]
Here we calculate the Lusztig algebra $\widetilde{A}(\sfG)$: we compute a general element of $(\End_K(T) \otimes A)^\sfG$ in Theorem \ref{Thm:LuzstigAlgS3}, and then the McKay quiver from it in Cor.~\ref{Cor:S3deg1}. Finally we compute the relations for Koszul algebras $A=T_K(V)/I$, where $I$ is an $S_3$-invariant ideal in $V \otimes V$.

Let now $A$ be a $\sfG-K$--algebra and $(a_{ij})$ a matrix over it. We  calculate the $\sfG$-invariant elements in $\End_K(T) \otimes A$, starting with the block $M^{22},$ as in Remark \ref{Rmk:Computationofdeg1}: Note that we just have to consider the action on the generators $s_1, s_2$.
\[
\begin{array}{|c|ccc|}
\hline
g& \rho_{2,1}(g)(ga_{ij})\rho_{2,1}(g)^{-1}&&\\
\hline
s_{1}&
\begin{pmatrix}
 0&1\\1&0
\end{pmatrix}
\begin{pmatrix}
 s_{1}a_{11}&s_{1}a_{12}\\
 s_{1}a_{21}&s_{1}a_{22}
\end{pmatrix}
\begin{pmatrix}
 0&1\\1&0
\end{pmatrix}&=&
s_{1}\begin{pmatrix}
 a_{22}&a_{21}\\
 a_{12}&a_{11}
\end{pmatrix}
\\
s_{2}&
\begin{pmatrix}
 0&\omega\\
 \omega^{2}&0
\end{pmatrix}
\begin{pmatrix}
 s_{2}a_{11}&s_{2}a_{12}\\
 s_{2}a_{21}&s_{2}a_{22}
\end{pmatrix}
\begin{pmatrix}
 0&\omega\\
 \omega^{2}&0
\end{pmatrix}&=&
s_{2}\begin{pmatrix}
 a_{22}&\omega^{2} a_{21}\\
\omega a_{12}& a_{11}
\end{pmatrix}
\\
\hline
\end{array}
\]
We investigate those matrices $(a_{ij})$ that are invariant and traceless. This amounts to the following equations according to the above table:
\[
\begin{array}{|c|ccc|}
\hline
g&&&\\
\hline
s_{1}&
\begin{pmatrix}
 0&1\\1&0
\end{pmatrix}
\begin{pmatrix}
 s_{1}a_{11}&s_{1}a_{12}\\
 s_{1}a_{21}&-s_{1}a_{11}
\end{pmatrix}
\begin{pmatrix}
 0&1\\1&0
\end{pmatrix}&=&
s_{1}\begin{pmatrix}
-a_{11}&a_{21}\\
 a_{12}&a_{11}
\end{pmatrix}
=
\begin{pmatrix}
a_{11}&a_{12}\\
 a_{21}&-a_{11}
\end{pmatrix}
\\
s_{2}&
\begin{pmatrix}
 0&\omega\\
 \omega^{2}&0
\end{pmatrix}
\begin{pmatrix}
 s_{2}a_{11}&s_{2}a_{12}\\
 s_{2}a_{21}&-s_{2}a_{11}
\end{pmatrix}
\begin{pmatrix}
 0&\omega\\
 \omega^{2}&0
\end{pmatrix}&=&
s_{2}\begin{pmatrix}
- a_{11}&\omega^{2} a_{21}\\
\omega a_{12}& a_{11}
\end{pmatrix}=
\begin{pmatrix}
a_{11}&a_{12}\\
 a_{21}&-a_{11}
\end{pmatrix}
\\
\hline
\end{array}
\]
These matrix equations say that $a_{11}$ is an $S_{3}$--anti-invariant and that
\[
a_{21}= s_{1}a_{12} =\omega s_{2}a_{12}\,.
\]
Therefore, $s_{1}-\omega s_{2}$ must annihilate $a_{12}$.

In summary, the invariant matrices are those of the form
\begin{align*}
\Hom_{K\sfG}(V,V\otimes A) &=
\left\{
\begin{pmatrix}
a+b&c\\
s_{1}c&a-b
\end{pmatrix}
\bigg|
a\in A^{S_{3}}, b\in A^{S_{3}}_{\sgn}, c\in A, (s_{1}-\omega s_{2})c=0
\right\}
\end{align*}
We now do the remaining entries. For $M^{21}$ we get
\[
\begin{array}{|c|ccc|}
\hline
g&&\Hom_{K\sfG}(K_{\triv},V\otimes A)&\\
\hline
s_{1}&
\begin{pmatrix}
 0&1\\1&0
\end{pmatrix}
\begin{pmatrix}
 s_{1}a_{10}\\
 s_{1}a_{20}
\end{pmatrix}
&=&
s_{1}\begin{pmatrix}
a_{20}\\
 a_{10}
\end{pmatrix}
=
\begin{pmatrix}
a_{10}\\
 a_{20}
\end{pmatrix}
\\
s_{2}&
\begin{pmatrix}
 0&\omega\\
 \omega^{2}&0
\end{pmatrix}
\begin{pmatrix}
 s_{2}a_{10}\\
 s_{2}a_{20}
\end{pmatrix}
&=&
s_{2}\begin{pmatrix}
 \omega a_{20}\\
 \omega^{2} a_{10}
\end{pmatrix}
=
\begin{pmatrix}
a_{10}\\
 a_{20}
\end{pmatrix}
\\
\hline
\end{array}
\]
This amounts to the conditions
\begin{align*}
a_{20}&=s_{1}a_{10}\\
a_{20}&=\omega^{2}s_{2}a_{10}
\end{align*}
Thus, $(s_{1}-\omega^{2}s_{2})a_{10}=0$ and $a_{20}=s_{1}a_{10}$.

Analogously, $M^{23}$
\[
\begin{array}{|c|ccc|}
\hline
g&&\Hom_{K\sfG}(K_{\sgn},V\otimes A)&\\
\hline
s_{1}&
\begin{pmatrix}
 0&1\\1&0
\end{pmatrix}
\begin{pmatrix}
 s_{1}a_{13}\\
 s_{1}a_{23}
\end{pmatrix}(-1)
&=&
-s_{1}\begin{pmatrix}
a_{23}\\
 a_{13}
\end{pmatrix}
=
\begin{pmatrix}
a_{13}\\
 a_{23}
\end{pmatrix}
\\
s_{2}&
\begin{pmatrix}
 0&\omega\\
 \omega^{2}&0
\end{pmatrix}
\begin{pmatrix}
 s_{2}a_{13}\\
 s_{2}a_{23}
\end{pmatrix}(-1)
%\begin{pmatrix}
% 1&0\\-1&-1
%\end{pmatrix}
&=&
-s_{2}\begin{pmatrix}
 \omega a_{23}\\
 \omega^{2} a_{13}
\end{pmatrix}
=
\begin{pmatrix}
a_{13}\\
 a_{23}
\end{pmatrix}
\\
\hline
\end{array}
\]
amounts to
\begin{align*}
a_{23}&=-s_{1}a_{13}\\
a_{23}&=-\omega^{2}s_{2}a_{13}\,.
\end{align*}
Thus, $(s_{1}-\omega^{2}s_{2})a_{13}=0$ and $a_{23}=-s_{1}a_{13}$.

 Finally, for $M^{12}$ and $M^{32}$ we get that
\[
\begin{array}{|c|ccc|}
\hline
g&&\Hom_{K\sfG}(V,K_{\triv}\otimes A)&\\
\hline
s_{1}&
\begin{pmatrix}
 s_{1}a_{01}&
 s_{1}a_{02}
\end{pmatrix}
\begin{pmatrix}
 0&1\\1&0
\end{pmatrix}
%\begin{pmatrix}
% 0&1\\1&0
%\end{pmatrix}
&=&
s_{1}\begin{pmatrix}
a_{02}&
 a_{01}
\end{pmatrix}
=
\begin{pmatrix}
a_{01}&
 a_{02}
\end{pmatrix}
\\
s_{2}&
\begin{pmatrix}
 s_{2}a_{01}&
 s_{2}a_{02}
\end{pmatrix}
\begin{pmatrix}
 0&\omega\\
 \omega^{2}&0
\end{pmatrix}
%\begin{pmatrix}
% 1&0\\-1&-1
%\end{pmatrix}
&=&
s_{2}\begin{pmatrix}
 \omega^{2}a_{02}&
\omega a_{01}
\end{pmatrix}
=
\begin{pmatrix}
a_{01}&
 a_{02}
\end{pmatrix}
\\
\hline
\end{array}
\]
yields
\begin{align*}
a_{02}&= s_{1}a_{01}\\
a_{02}&=\omega s_{2}a_{01}\,.
\end{align*}
Thus, $(s_{1}-\omega s_{2})a_{01}=0$ and $a_{02}= s_{1}a_{01}$.
\[
\begin{array}{|c|ccc|}
\hline
g&&\Hom_{K\sfG}(V,K_{\sgn}\otimes A)&\\
\hline
s_{1}&
-\begin{pmatrix}
 s_{1}a_{31}&
 s_{1}a_{32}
\end{pmatrix}
\begin{pmatrix}
 0&1\\1&0
\end{pmatrix}
%\begin{pmatrix}
% 0&1\\1&0
%\end{pmatrix}
&=&
-s_{1}\begin{pmatrix}
a_{32}&
 a_{31}
\end{pmatrix}
=
\begin{pmatrix}
a_{31}&
 a_{32}
\end{pmatrix}
\\
s_{2}&
-\begin{pmatrix}
 s_{2}a_{31}&
 s_{2}a_{32}
\end{pmatrix}
\begin{pmatrix}
 0&\omega\\
 \omega^{2}&0
\end{pmatrix}
%\begin{pmatrix}
% 1&0\\-1&-1
%\end{pmatrix}
&=&
-s_{2}\begin{pmatrix}
 \omega^{2}a_{32}&
\omega a_{31}
\end{pmatrix}
=
\begin{pmatrix}
a_{31}&
 a_{32}
\end{pmatrix}
\\
\hline
\end{array}
\]
yields
\begin{align*}
a_{32}&= -s_{1}a_{31}\\
a_{32}&=-\omega s_{2}a_{31}
\end{align*}
Thus, $(s_{1}-\omega s_{2})a_{31}=0$ and $a_{32}= -s_{1}a_{31}$.

We have now established the following result.
\begin{Thm} \label{Thm:LuzstigAlgS3}
A matrix in $\End_{K}(T)\otimes A$ is invariant if, and only if, it is of the following form
\begin{align*}
\begin{pmatrix}
a_{00}&a_{01}&s_{1}a_{01}&a_{03}\\
a_{10}&a+b&c&a_{13}\\
s_{1}a_{10}&s_{1}c&a-b&-s_{1}a_{13}\\
a_{30}&a_{31}&-s_{1}a_{31}&a_{33}
\end{pmatrix}
\end{align*}
where
\begin{itemize}
\item $a,a_{00}, a_{33}\in A^{S_{3}}$ are invariants,
\item $b, a_{03}, a_{30}\in A^{S_{3}}_{\sgn}$ are anti--invariants,
\item $u\in\{a_{10},a_{13}\}\subset A$ satisfies $(s_{1}-\omega^{2}s_{2})u=0$,
\item $v\in \{c, a_{01},a_{31}\}\subset A$ satisfies $(s_{1}-\omega s_{2})v=0$.
\end{itemize}
\end{Thm}
 
Now we are interested in the degree $1$ part of $\widetilde{A}(\sfG,T)$ (see Section  \ref{Sub:QuiverRelationsKoszul} on the assumptions on $A$ and grading), as this encodes the McKay quiver of $\sfG$ and will provide us with the relations. Assume that $u,v$ as in the Theorem are in a copy of $V=Kx\oplus Ky \subset A$.
Given the action of $S_{3}$ on $V$ as specified above, we have that
$u=\alpha x +\beta y\in V$ satisfies $(s_{1}-\omega^{2}s_{2})u=0$ if, and only if,
\begin{align*}
0=(s_{1}-\omega^{2}s_{2})(\alpha x +\beta y)&=(x,y)
\left(\begin{pmatrix}
0&1\\
1&0
\end{pmatrix}
-\begin{pmatrix}
0&1\\
\omega&0
\end{pmatrix}
\right)
\begin{pmatrix}
\alpha\\
\beta
\end{pmatrix}
=
(x,y)\begin{pmatrix}
0&0\\
1-\omega&0
\end{pmatrix}
\begin{pmatrix}
\alpha\\
\beta
\end{pmatrix}
\end{align*}
which means that $\alpha=0$, thus $u=\beta y$ and $s_{1}(u)=\beta x$.

Similarly, the equation $(s_{1}-\omega s_{2})v=0$ for $v=\gamma x +\delta y$ amounts to
$v=\gamma x$ and $s_{1}(v)=\gamma y$.

\begin{corollary} \label{Cor:S3deg1}
If the matrix in the theorem is in $\End_{K}(T)\otimes V\subset \End_{K}(T)\otimes A$, then it is of the form  
\begin{align} \label{Eq:MS3}
M=\begin{pmatrix}
0&bx&by&0\\
ay&0&ex&dy\\
ax&ey&0&-dx\\
0&cx&-cy&0
\end{pmatrix}
\end{align}
for $a,b,c,d,e \in K$, whence the degree $2$ part of $(\End_{K}(T)\otimes A)^\sfG$ is of the form

\begin{align*}
MM'&=\begin{pmatrix}
0&bx&by&0\\
ay&0&ex&dy\\
ax&ey&0&-dx\\
0&cx&-cy&0
\end{pmatrix}
\begin{pmatrix}
0&b'x&b'y&0\\
a'y&0&e'x&d'y\\
a'x&e'y&0&-d'x\\
0&c'x&-c'y&0
\end{pmatrix}\\
&=
\begin{pmatrix}
ba'(xy+yx)&be'y^{2}&be'x^{2}&bd'(xy-yx)\\
ea'x^{2}&(ab'+dc')yx+ee'xy&(ab'-dc')y^{2}&-ed'x^{2}\\
ea'y^{2}&(ab'-dc')x^{2}&(ab'+dc')xy+ee'yx&ed'y^{2}\\
ca'(xy-yx)&-ce'y^{2}&ce'x^{2}&cd'(xy+yx)
\end{pmatrix}
\end{align*}
for some $a',b',c',d',e'\in K$.
\end{corollary}

 Note that the $M$ from \eqref{Eq:MS3} provides us with the McKay quiver of $S_3$: 
 \[
\begin{tikzpicture}
\node at (4,0) {\begin{tikzpicture} 
\node (C1) at (0,0)  {$\rho_{\mathrm{triv}}$};
\node (C2) at (2,0)  {$\rho_{2,1}$};
\node (C3) at (4,0)  {$\rho_{\mathrm{\sgn}}$};

\draw [->,bend left=20,looseness=1,pos=0.5] (C1) to node[above]  {A} (C2);
\draw [->,bend left=20,looseness=1,pos=0.5] (C2) to node[below] {B} (C1);
\draw[->]  (C2) to [out=120,in=60, loop,looseness=10,pos=0.5] node[above] {E} (C2);

\draw [->,bend left=20,looseness=1,pos=0.5] (C3) to node[below]  {D} (C2);
\draw [->,bend left=20,looseness=1,pos=0.5] (C2) to node[above] {C} (C3);

\end{tikzpicture}}; 
\end{tikzpicture}
\] 
We label the arrows by the corresponding scalars in \eqref{Eq:MS3}, so e.g., the arrow $E$ corresponds to $$\begin{pmatrix} 0 & 0 & 0 & 0\\ 0 & 0 &ex& 0 \\ 0 & ey & 0 & 0 \\ 0 & 0 & 0 & 0 \end{pmatrix}.$$

Now consider $V \otimes_K V = \Sym_K(V) \oplus K_{\sgn}$ as $S_3$-representations (here: $\Sym_K(V)=\rho_{2,1}$). We see that we get generators $\langle x^2, y^2, xy+yx \rangle$ for $\Sym_K(V)$ and $\langle xy-yx \rangle$ for $K_{\sgn}$. The $S_3$-invariant subspaces of $V \otimes V$ are then given by sums of the ideals 
$$I_1=\langle xy+yx \rangle, \ I_2=\langle x^2,y^2 \rangle,  \ I_3=\langle xy-yx \rangle \ .$$
Thus we can calculate the relations for $\widetilde{A}(\sfG)$, where $A=T_K(V)/I$, where $I=\sum_{i=1}^3\varepsilon_{i}I_i$ and $\varepsilon \in \{ 0 , 1\}$. This gives $8$ possibilities for $I$: 
\begin{enumerate}[leftmargin=*]
\item $I=\langle 0 \rangle$: Then $A=T_K(V)$ and there are no relations in $A*\sfG$. This means that $A*\sfG$ is Morita equivalent to the basic algebra $KQ$, where $Q$ is the McKay quiver of $S_3$. 
\item $I=I_1$: Then $A_1=T_K(V)/I_1$ and we immediately see the relations $BA=CD=0$. The middle representation yields the relation $AB+DC+2E^2=0$.
\item $I=I_2$: Then $A_2=T_K(V)/I_2$ and we see that then immediately $BE=ED=CE=EA=0$ and from the middle that $AB=DC$. Note that $A_2$ is a quadratic monomial algebra and is Koszul by \cite{Froeberg}, and the same holds for $\widetilde{A_2}(\sfG)$ by Thm.~\ref{Thm:LusztigAlgKoszul}.
\item $I=I_3$: Then $A_3=\Sym_K(V)$ and $A_3*\sfG$ is the usual skew group ring. From the commutativity relations, we immediately see that $BD=CA=0$. The middle part of $MM'$ yields the relation $AB+DC-2E^2=0$. Note that the algebra $A$ in \cite[Prop.~4.11 (i)]{DaFI} is $\widetilde{A_3}(\sfG)/\widetilde{A_3}(\sfG)e_{\rho_{\mathrm{\sgn}}}\widetilde{A_3}(\sfG)$, and we get the relation $AB=2E^2$ by deleting all paths going through the vertex $\rho_{\mathrm{\sgn}}$. This is the relation calculated in loc.cit. by other means.
\item $I=I_1+I_2$: Then $A_4=\bigwedge_K(V)$ is the exterior algebra. The immediate relations are $BA=BE=EA=ED=CE=CD=0$, and from the middle part we obtain the two relations $AB+E^2=0$ and $DC+E^2=0$. 
\item $I=I_1+I_3$: Here note that $I=\langle xy,yx \rangle$, and thus $A_5$ is a quadratic monomial algebra and Koszul. Note that $A_5$ is the Koszul dual of $A_2$. The relations in $\widetilde{A_5}(\sfG)$ are $BA=BD=CA=CD=0$ and from the middle $2 \times 2$-block $AB+DC=0$ and $E^2=0$. Moreover $\widetilde{A_5}(\sfG)$ is Koszul dual to $\widetilde{A_2}(\sfG)$.
\item $I=I_2+I_3$: $A_6=T_K(V)/I$ is Koszul dual to $A_1$. For the Lusztig algebra $\widetilde{A_6}(\sfG)$ we obtain relations $BE=EA=ED=CE=CA=BD=0$ and $AB-E^2=0$ and $DC-E^2=0$.
\item $I=I_1+I_2+I_3$: Here $I=\langle x^2,y^2,xy,yx \rangle$ is quadratic and $A_7=T_K(V)/I$ is finite dimensional with Koszul dual $A=T_K(V)$. The Lusztig algebra $\widetilde{A_7}(\sfG)$ is then $KQ$ modulo all paths of length $2$, which is also finite dimensional. \end{enumerate}
 \end{example}

\begin{example} \label{Ex:S4}
Let $\sfG=S_4$ be acting on $V=Kx\oplus Ky \oplus Kz$ as a reflection group, i.e., consider $G(1,1,4)$. Using a computer algebra system (e.g.~Macaulay2) one can calculate the degree one part of of $A*\sfG$, yielding a $10\times10$ matrix:
\begin{equation*}
  \scalebox{0.45}{\rotatebox{0}{
$\begin{pmatrix}
       0&\left(2\,x+y+z\right)a^*&\left(x+2\,y+z\right)a^*&\left(x+y+2\,z\right)a^*&0&0&0&0&0&0\\
       xa&\left(2\,x-y-z\right)f&\left(-x-2\,y-z\right)f&\left(-x-y-2\,z\right)f&\left(-x-2\,y-z\right)b^{*}&\left(-y+z\right)b^{*}&\left(-x-3\,y\right)g^*&\left(-y+z\right)g^*&\left(-x-3\,z\right)g^*&0\\
       ya&\left(-2\,x-y-z\right)f&\left(-x+2\,y-z\right)f&\left(-x-y-2\,z\right)f&\left(-x+z\right)b^{*}&\left(-2\,x-y-z\right)b^{*}&\left(3\,x+y\right)g^*&\left(-y-3\,z\right)g^*&\left(-x+z\right)g^*&0\\
       za&\left(-2\,x-y-z\right)f&\left(-x-2\,y-z\right)f&\left(-x-y+2\,z\right)f&\left(x+2\,y+z\right)b^{*}&\left(2\,x+y+z\right)b^{*}&\left(-x+y\right)g^*&\left(3\,y+z\right)g^*&\left(3\,x+z\right)g^*&0\\
       0&\left(-2\,x-y-z\right)b&\left(-x+y+2\,z\right)b&\left(-x+2\,y+z\right)b&0&0&\left(x+z\right)c^{*}&\left(-y+z\right)c^{*}&\left(x+y\right)c^{*}&0\\
       0&\left(x-y+2\,z\right)b&\left(-x-2\,y-z\right)b&\left(2\,x-y+z\right)b&0&0&\left(-y-z\right)c^{*}&\left(x+y\right)c^{*}&\left(-x+z\right)c^{*}&0\\
       0&\left(3\,y+z\right)g&\left(-3\,x-z\right)g&\left(-x+y\right)g&\left(3\,x+z\right)c&\left(-3\,y-z\right)c&\left(x+y-2\,z\right)e&\left(2\,x+y+z\right)e&\left(-x-2\,y-z\right)e&zd^{*}\\
       0&\left(-y+z\right)g&\left(x+3\,z\right)g&\left(-x-3\,y\right)g&\left(x+3\,z\right)c&\left(2\,x+3\,y+3\,z\right)c&\left(x+y+2\,z\right)e&\left(-2\,x+y+z\right)e&\left(-x-2\,y-z\right)e&xd^{*}\\
       0&\left(y+3\,z\right)g&\left(-x+z\right)g&\left(-3\,x-y\right)g&\left(3\,x+2\,y+3\,z\right)c&\left(y+3\,z\right)c&\left(-x-y-2\,z\right)e&\left(-2\,x-y-z\right)e&\left(x-2\,y+z\right)e&{-yd^{*}}\\
       0&0&0&0&0&0&\left(-x-y-2\,z\right)d&\left(-2\,x-y-z\right)d&\left(x+2\,y+z\right)d&0\end{pmatrix}$
  }}
  \end{equation*}
This yields the McKay quiver (which is luckily the same as we calculated earlier in \eqref{Eq:S4-McKayQuiver}), now with labelled arrows:
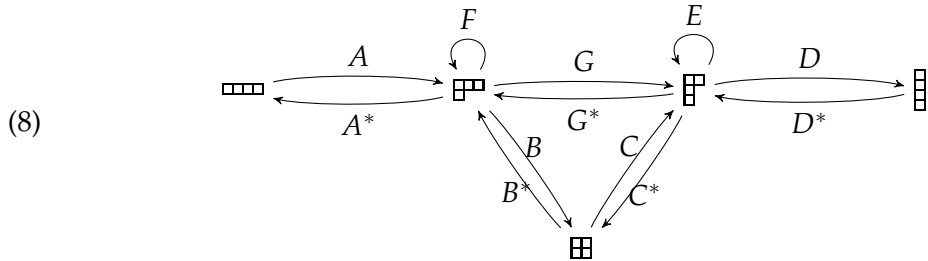
\begin{equation} 
{\begin{tikzpicture}[baseline=(current  bounding  box.center)] 
\ytableausetup{centertableaux,boxsize=0.3em}
\node (triv) at (0,0) {$\ydiagram{4}$};

\node (R1) at (3,0) {$\ydiagram{3,1}$} ;
\node (R2) at (6,0)  {$\ydiagram{2,1,1}$};
\node (det) at (9,0)  {$\ydiagram{1,1,1,1}$};
\node (W) at (4.5,-2.1)  {$\ydiagram{2,2}$};

\draw [->,bend left=15,looseness=0.5,pos=0.5] (triv) to node[]  [above]{$A$} (R1);
\draw [->,bend left=15,looseness=0.5,pos=0.5] (R1) to node[]  [below]{$A^*$} (triv);

\draw [->,bend left=10,looseness=0.5,pos=0.5] (R1) to node[] [above] {$G$} (R2);
\draw [->,bend left=10,looseness=0.5,pos=0.5] (R2) to node[] [below] {$G^*$} (R1);

\draw [->,bend left=15,looseness=0.5,pos=0.5] (R2) to node[] [above] {$D$} (det);
\draw [->,bend left=15,looseness=0.5,pos=0.5] (det) to node[] [below] {$D^*$} (R2);

\draw [->,bend left=10,looseness=0.5,pos=0.5] (R1) to node[] [above] {$B$} (W);
\draw [->,bend left=10,looseness=0.5,pos=0.5] (W) to node[] [below] {$B^*$} (R1);
\draw [->,bend left=10,looseness=0.5,pos=0.5] (R2) to node[] [below] {$C^*$} (W);
\draw [->,bend left=10,looseness=0.5,pos=0.5] (W) to node[] [above] {$C$} (R2);

\draw [->,in=120, out=60,loop,looseness=5,pos=0.5] (R1) to node[] [above] {$F$} (R1);
\draw [->,in=120, out=60,loop,looseness=4,pos=0.5] (R2) to node[] [above] {$E$} (R2);

\end{tikzpicture}} 
\end{equation}

 Here we just give the relations for the Lusztig algebra $\widetilde{A}(\sfG)$ associated to the skew group ring $S*\sfG$, that is, we set $A=\Sym_K(V)$. A computation with Macaulay2 shows that in total there are $12$ quadratic relations, with at most $4$ terms:
\begin{align*}
\text{1 term:} \quad& GA=A^*G^*=G^*D^*=DG=0 \ , \\
\text{2 terms:} \quad& BG^*+3C^*E=-GB^*+3EC^*=2FB^*+G^*C=-2BF+C^*G=0 \ , \\
\text{3 terms:} \quad & 2FG^*+4B^*C^*-G^*E=-2GF+4CB+EG=0 \ , \\
\text{4 terms:} \quad & 3GG^*+12CC^*+3E^2+16D^*D=-32AA^*+12F^2+4B^*B+3G^*G=0 \ . \\
\end{align*}
\end{example}

\begin{Bem}
The last three examples suggest that in general the Lusztig algebra $\widetilde{S}(\sfG)$ does not have the structure of a preprojective algebra. There are also examples of $\widetilde{S}(\sfG)$ for $\sfG$ the product of cyclic groups that are not isomorphic to a (higher) preprojective algebra, see \cite{AmiotIyamaReiten,Thibault}. It would be interesting to find a pattern for the relations for $\sfG=S_n$ and other reflection groups. 
\end{Bem}

%\bibliographystyle{plain}
%\bibliography{mybibliography}

\end{document}